\documentclass[11pt,a4paper,leqno]{article}

\usepackage{amsmath}  
\usepackage{amsthm}		
\usepackage{bbm}      
\usepackage{amssymb}  
\usepackage{ifthen}   
\usepackage[T1]{fontenc}  
\usepackage[latin1]{inputenc}  
\usepackage{framed} 

\usepackage{paralist}

\advance\topmargin by -1cm
\advance\oddsidemargin by -0.5cm
\advance\evensidemargin by -0.5cm
\newtheorem{thm}{Theorem}[section]
\newtheorem{prop}[thm]{Proposition}
\newtheorem{lem}[thm]{Lemma}

\theoremstyle{definition}

\newtheorem{example}[thm]{Example}
\newtheorem{remark}[thm]{Remark}

\newcommand{\cst}{\texttt{c{\!}s{\!}t}\,}
\renewcommand{\d}{\partial}

\newcommand{\eps}{\varepsilon}
\let\oddtheta\theta
\let\theta\vartheta
\let\vartheta\oddtheta
\let\oddphi\phi
\let\phi\varphi
\let\varphi\oddphi
\let\oddrho\rho
\let\rho\varrho
\let\varrho\oddrho
\newcommand{\B}{\mathbb B}

\newcommand{\N}{\mathbb N}
\newcommand{\Q}{\mathbb Q}
\newcommand{\R}{\mathbb R}

\renewcommand{\L}{\mathbb L}
\newcommand{\E}{\mathbb E} 
\renewcommand{\P}{\mathbb P} 

\newcommand{\cB}{\mathcal B}

\newcommand{\cD}{\mathcal D}
\newcommand{\cE}{\mathcal E}
\newcommand{\cF}{\mathcal F}
\newcommand{\cG}{\mathcal G}

\newcommand{\cI}{\mathcal I}

\newcommand{\cL}{\mathcal L}

\newcommand{\cN}{\mathcal N}

\newcommand{\cS}{\mathcal S}

\newcommand{\tE}{\mathtt E}

\newcommand{\abs}[1]{\left|#1\right|}
\newcommand{\norm}[1]{\left\|#1\right\|}

\newcommand{\inpro}[2]{\left\langle#1,#2\right\rangle}

\title{Local asymptotic normality for shape and periodicity of a signal in the drift of a degenerate diffusion with internal variables}
\author{Simon Holbach\footnote{Fakult{\"a}t f{\"u}r Mathematik, Universit{\"a}t Bielefeld, Postfach 10 01 31, 33501 Bielefeld, Germany, e-mail: sholbach$@$math.uni-bielefeld.de. This article features results from the author's PhD thesis \cite{ICHDiss} at Johannes Gutenberg-Universit{\"a}t Mainz.}}

\begin{document}
\maketitle
\begin{abstract}
Taking a multidimensional time-homogeneous dynamical system and adding a randomly perturbed time-dependent deterministic signal to some of its components gives rise to a high-dimensional system of stochastic differential equations which is driven by possibly very low-dimensional noise. Equations of this type are commonly used in biology for modeling neurons or in statistical mechanics for certain Hamiltonian systems. Assuming that the signal depends on an unknown shape parameter $\theta$ and also has an unknown periodicity $T$, we prove Local Asymptotic Normality (LAN) jointly in $\theta$ and $T$ for the statistical experiment arising from (partial) observation of this diffusion in continuous time. The local scale turns out to be $n^{-1/2}$ for $\theta$ and $n^{-3/2}$ for $T$ which generalizes known results for simpler systems.
\end{abstract}

\small{\textbf{Keywords:} local asymptotic normality, parametric signal estimation, degenerate diffusion, periodic drift

}

\section{Introduction of the model and the problem}\label{sect:intro}

Let $U\subset\R^{N+L}$ be a $\sigma$-compact set and let $f\colon U \to \R^N$ and $g\colon U \to \R^L$ be locally Lipschitz continuous functions. Finally, let $S\colon[0,\infty)\to\R^N$ be a continuous periodic signal and consider the deterministic dynamical system
\begin{align}\label{eq:ODE}
\begin{split}
dX_t &= f(X_t,Y_t)dt + S(t)dt, \\
dY_t &= g(X_t,Y_t)dt.
\end{split}
\end{align}
This system is divided into two groups of variables: The $N$ components of $X$ whose dynamics depend directly on the signal and the $L$ components of $Y$ which are affected by the signal only indirectly through the influence of $X$. Intuitively speaking, we can think of \eqref{eq:ODE} as a dynamical system with no intrinsic time-inhomogeneity which then receives an additional time-dependent \emph{external input} $S$ in some of its variables, while the remaining variables merely describe an interior mechanism. This is why we sometimes refer to $X$ as the \emph{adjustable variable(s)} and $Y$ as the \emph{internal variable(s)}. Note that the only source of time-inhomogeneity is indeed the signal -- if the system receives constant external input $S\equiv c \in \R^N$ (or none at all, i.e.\ $c=0$), it is homogeneous in time. Systems of this kind frequently arise in the context of neuroscience and statistical mechanics (see Examples \ref{ex:HH} and \ref{ex:rotors} below).

We construct a stochastic model by following the idea that the signal is not actually received in its original shape, but is subject to random perturbations by external noise (i.e.\ noise that is independent of the rest of the system). To take account of this notion, it seems natural to substitute the signal term $S(t)dt$ in \eqref{eq:ODE} with the increment $dZ_t$ of a process taking values in a closed set $U'\subset\R^N$ and satisfying an SDE of the type
\[
dZ_t = [S(t)+b(Z_t)]dt + \sigma(Z_t)dW_t,
\]
where $W$ is an $M$-dimensional standard Brownian Motion, while $b\colon U' \to \R^N$ and $\sigma\colon U' \to \R^{N\times M}$ are locally Lipschitz continuous drift and volatility functions. Note that this SDE can be viewed as a generalized Orstein-Uhlenbeck type process with time-dependent mean-reversion level (think of $b(Z_t)=-\beta Z_t$ with $\beta\in(0,\infty)$). A particularly prominent special case is the classical signal in noise model (take $M=N=1$, $b\equiv 0$, and $\sigma\equiv 1$, see for example \cite[Example I.7.3, Chapter III.5]{Ibra}), which arises in a wide variety of fields including communication, radiolocation, seismic signal processing, or computer-aided diagnosis and has been the subject of extensive study.

Perturbing $S(t)$ randomly in this way leads to the stochastic dynamical system
\begin{align}\label{eq:SDE}
	\begin{split}
		dX_t &= f(X_t,Y_t)dt + dZ_t, \\
		dY_t &= g(X_t,Y_t)dt, \\
		dZ_t &= [S(t)+b(Z_t)]dt + \sigma(Z_t)dW_t,
	\end{split}
\end{align}
with state space
\[
	\tE:=U \times U'\subset\R^{N+L+N}.
\]
This system can be thought of as degenerate in the following sense: Firstly, the equation for $Y$ does not incorporate the driving Brownian Motion $W$ explicitly, making it rather unclear which effect noise has on these components. Secondly, the dimension $M$ of the driving Brownian Motion can (and will usually) be much lower than the dimension $N+L+N$ of the system. This is why we call a stochastic process satisfying a system of stochastic differential equations of the type \eqref{eq:SDE} a \emph{degenerate diffusion with internal variables and randomly perturbed time-inhomogeneous deterministic input}.

We now have three groups of variables: The entirely autonomous external input governed by $dZ_t$ (the "noisy signal"), the components of $X$ that are directly adjusted by the noisy signal, and the components of the internal variable $Y$ whose dynamics are only indirectly affected by noise, since the respective differential equations incorporate neither $Z$ nor the driving Brownian Motion $W$ explicitly. Note that for this reason $Y$ is conditionally deterministic given $X$ and has continuously differentiable trajectories.

The system \eqref{eq:SDE} is a generalization of the one introduced in equation (18) of Section 4.1 of \cite{HLT2}, which is a probabilistic version of a class of dynamical systems that are well-known in the mathematical modeling of neurons (see Example \ref{ex:HH} below). In \cite{ICHHarris} (which can be viewed as a companion article to the present one), we study the model \eqref{eq:SDE} from a purely probabilistic standpoint and use methods from \cite{HLT3} to discuss sufficient conditions for the process $(X,Y,Z)$ to be positive Harris recurrent. Before we explain the focus of the current article, let us introduce two major examples.

\begin{example}\label{ex:HH}
	Let $N=1$, $L=3$, $U=\R\times[0,1]^3$ and consider the coefficient functions
	\[
		f(x,y)=-36y_1^4(x+12)-120y_2^3y_3(x-120)-0.3(x-10.6)
	\]
	and
	\[
	g(x,y)=\begin{pmatrix} \alpha_1(x)(1-y_1)-\beta_1(x)y_1 \\ \alpha_2(x)(1-y_2)-\beta_2(x)y_2 \\ \alpha_3(x)(1-y_3)-\beta_3(x)y_3	\end{pmatrix}
	\]
	with
	\[
	\begin{array}{llllll}
	\alpha_1(x) &=& \begin{cases} \frac{0.1-0.01x}{\exp(1-0.1x)-1}, & x\neq 10, \\ 0.1, & \text{else},\end{cases}& \beta_1(x)& = & 0.125\exp(-x/80),  \vspace{3mm} \\
	\alpha_2(x) &=& \begin{cases} \frac{2.5-0.1x}{\exp(2.5-0.1x)-1}, & x\neq 25, \\ 1, & \text{else},\end{cases}& \beta_2(x)& = & 4\exp(-x/18), \vspace{2.5mm} \\
	\alpha_3(x) &=& 0.07\exp(-x/20),& \beta_3(x)& = & \frac{1}{\exp(3-0.1x)+1}. \\
	\end{array}
	\]
	for all $(x,y)=(x,y_1,y_2,y_3)^\top\in U$. The corresponding dynamical system \eqref{eq:ODE} is known as the \emph{Hodgkin-Huxley system} and it was first introduced by Hodgkin and Huxley in 1952 (see \cite{HoHu}, note however that we use the slightly different model constants from \cite{Izhi}) with the aim of describing the initiation and propagation of action potentials in the cell membrane of a neuron in response to an external stimulus. While $X$ is the membrane potential itself (usually labeled $V$ in the literature), the internal variables $Y_1$, $Y_2$, and $Y_3$ (commonly denoted by $n$, $m$, and $h$) correspond to the ionic mechanism underlying its evolution. The two predominant ion currents in the cell membrane are import of sodium $Na^+$ and export of potassium $K^+$ through the membrane. Each of the internal variables signifies the probability that a specific type of gate in the respective ion channel is open at a given time. It is for this reason that $n$, $m$, and $h$ are often called gating variables. In the context of this model, the signal $S$ represents the dendritic input which the neuron receives from a large number of other neurons, transported by an even larger number of synapses located on the respective dendritic tree. The resulting "total dendritic input" can then be thought of as an average of interdependent and repeating similar currents, which is why $S$ is usually assumed to be periodic (or even constant). When modeling neurons, particular interest lies in the typical spiking behaviour of the membrane potential, a feature that is commonly agreed upon to be adequately described by the Hodgkin-Huxley model. For a more detailed modern introduction, interpretation, and an in-depth comparison with other neuron models, see for example \cite{Izhi} and \cite{Dest}.
	
	Adding noise in the sense of \eqref{eq:SDE} by choosing	$\sigma\in C^\infty(U')$ and $b(Z_t)=- \beta Z_t$ with $\beta\in(0,\infty)$, we acquire the so-called \emph{stochastic Hodgkin-Huxley model (with mean reverting Ornstein-Uhlenbeck type input)}. It was first introduced and studied by H{\"o}pfner, L{\"o}cherbach, and Thieullen in the series of the three papers \cite{HLT1}, \cite{HLT2}, and \cite{HLT3}. The constant $\beta$ is determined by the so-called time constant of the membrane which represents spontaneous voltage decay not related to the input. For many types of neurons, the time constant is known from experiments (see \cite{Ditlevsen}). A degree of freedom lies in the choice of the volatility $\sigma$ which reflects the nature of the influence of noise. In the past, mean reverting Ornstein-Uhlenbeck type equations with various volatilities have been used to model the membrane potential itself (see for example \cite{Lansky} or \cite{HoBio}), and in a sense our stochastic Hodgkin-Huxley model can be viewed as a refinement of this kind of model. If $\sigma$ is Lipschitz continuous, existence of a unique non-exploding strong solution taking values in $\tE=\R\times [0,1]^3\times U'$ follows from the same arguments as in \cite[Proposition 1]{HLT1} and \cite[Proposition 2]{HLT2}.
	
	Analogously, one can introduce stochastic versions of simpler neuron models such as the FitzHugh-Nagumo model (see \cite[equations (4.11) and (4.12)]{Izhi}) or the Morris-Lecar model (see \cite{Morris} or, for a modern version, \cite{Rinzel}).
\end{example}

\begin{example}\label{ex:rotors}
	Systems of coupled oscillators are particularly intuitive Hamiltonian systems and several different stochastic models have been subject to research in the past (see e.g\ \cite{Hairer}, \cite{Cuneo2}, \cite{Rey-Bellet}, \cite{Rotoren2018}). The following example is inspired by the model from \cite{Cuneo} to which we add a time-inhomogeneity and the corresponding external variables.
	
	Let us think of three rotors, each given by their angle $q_i(t)\in\R$ and momentum $p_i(t)\in\R$ at the time $t\in[0,\infty)$ for each $i\in\{1,2,3\}$. Assuming their respective masses to be all equal to $1$ and not taking into account units, the laws of classical mechanics imply
	\begin{equation}\label{eq:rotors0}
	\dot q_i = p_i \quad \text{for all $i\in\{1,2,3\}$.}
	\end{equation}
	We suppose that these rotors are coupled in row, i.e.\
	\begin{align}\label{eq:rotors}
		\begin{split}
			\dot p_1 &= w_1(q_2-q_1)-u_1(q_1), \\
			\dot p_2 &= -[w_1(q_2-q_1) + w_3(q_2-q_3)]-u_2(q_2), \\
			\dot p_3 &= w_3(q_2-q_3)-u_3(q_3),
		\end{split}
	\end{align}
	where $w_1,w_2,w_3\colon\R\to\R$ and $u_1,u_2,u_3\colon\R\to\R$ are related to interaction potentials and pinning potentials, respectively. A classical model is the one that arises if we let one or both of the outer rotors receive external torques and interact with Langevin type heat baths. In order to give a mathematical description of this, we fix	$i\in\{1,3\}$ for the remainder of this paragraph. Applying an external time-dependent torque $S_i\colon[0,\infty)\to\R$ to the $i$-th rotor means expanding the equation for $p_i$ to
	\[
	dp_i = \left[w_i(q_2-q_i)-u_i(q_i)\right]dt +S_i dt,
	\]
	which turns \eqref{eq:rotors0} and \eqref{eq:rotors} into a system like \eqref{eq:ODE}. On top of that, we want to add interaction with a heat bath, i.e.\ for a temperature $\tau_i\in(0,\infty)$ and a dissipation constant $\delta_i\in(0,\infty)$, the equation for $p_i$ is further expanded to
	\begin{align*}
		dp_i &= \left[w_i(q_2-q_i)-u_i(q_i)\right]dt + S_idt - \delta_i p_i dt + \sqrt{2\delta_i\tau_i}dW^{(i)}_t \\
			&=\left[w_i(q_2-q_i)-u_i(q_i)-\delta_i p_i\right]dt  + \left[S_i dt+ \sqrt{2\delta_i\tau_i}dW^{(i)}_t\right],
	\end{align*}
	where the last term in parentheses is the total sum of external influences. Following the spirit of \eqref{eq:SDE}, we may replace this term with the increments of a more general random perturbation of the torque: We take
	\[
	dp_i =\left[w_i(q_2-q_i)-u_i(q_i)-\delta_i p_i\right]dt  + dZ^{(i)}_t
	\]
	with
	\[
	dZ^{(i)}_t=\left[S_i(t)+b_i(Z^{(i)}_t)\right]dt + \sigma_i(Z^{(i)}_t)dW^{(i)}_t
	\]
	for some volatility $\sigma_i\colon\R\to\R$ and a drift $b_i\colon\R\to\R$. What we end up with is indeed a degenerate diffusion with internal variables and randomly perturbed time-inhomogeneous deterministic input as in \eqref{eq:SDE}. If only the first rotor in the chain receives an external input, the dimensions are $M=N=1$ and $L=5$, $U=\R^6$, $U'=\R$. If both of the outer rotors receive an external input, the dimensions are $M=N=2$ and $L=4$, $U=\R^6$, $U'=\R^2$.
\end{example}

In this article, we want to study a statistical model in which the deterministic signal $S$ depends on a set of parameters. More precisely, we assume that there is an open set $\Theta\subset\R^D$ such that
\[
S=S_{(\theta,T)} \quad \text{with $(\theta,T)\in\Theta\times(0,\infty)$},
\]
where $T$ is the signal's periodicity and $\theta$ is a $D$-dimensional shape parameter. A natural goal is to estimate $\theta$ and $T$ simultaneously from continuous observation of the process. However, observing the process $(X,Y,Z)$ entirely may not make sense in many models: The external variable $Z$ can be of a rather abstract nature and, for example, in the Hodgkin-Huxley model from Example \ref{ex:HH} the only variable that is arguably observable is the membrane potential $X$. In spite of that, Section \ref{sect:observe} shows:
\begin{framed}
	\noindent\textbf{Result 1.} As long as the initial configuration $(X_0,Y_0,Z_0)$ is deterministic and known, it does not matter whether we can observe the entire process $(X,Y,Z)$, only the adjustable variable $X$, or only the external variable $Z$.
\end{framed}
\noindent This is the content of Remark \ref{rem:observe} and Proposition \ref{prop:likelihoods}. Since $Z$ is the most convenient process to handle statistically among all of these, our considerations in the sequel are confined to this external variable. Being able to relate statistical problems entirely to $Z$ means that as long as this variable fits our setting, we can treat any example of \eqref{eq:SDE} (including in particular those that were introduced in Examples \ref{ex:HH} and \ref{ex:rotors}). In Section \ref{sect:LAN}, we prove an LAN result for the external variable (Theorem \ref{thm:LAN}), generalizing \cite[Theorem 2.3]{ICH} in which we only treated the case $M=N=1$. This can then be combined with the previous results in order to obtain:
\begin{framed}
	\noindent\textbf{Result 2.} Under reasonable regularity conditions on the parametrization and under some non-degeneracy and ergodicity of the external variable $Z$, the sequence of statistical experiments corresponding to continuous observation of $(X,Y,Z)$ over growing time intervals $[0,n]$ for $n\to\infty$ has the LAN property. The local scales are identified as $n^{-1/2}$ for the shape and $n^{-3/2}$ for the periodicity.
\end{framed}
\noindent The rigorous and precise corresponding statement is Theorem \ref{thm:LANX}. It allows for application to simultaneous estimation of shape and periodicity, as under LAN we can use H\'{a}jek's Convolution Theorem and the Local Asymptotic Minimax Theorem in order to establish optimality for estimators when the rescaled estimation errors are stochastically asymptotically equivalent to the central statistic of the experiment (see \cite{LeCam}, \cite {Davies}, \cite{Kut} or \cite{HoBo} for a detailed presentation of the relevant theory).

\section{Main results and applications}\label{sect:results}

First, let us recall and collect the basic assumptions that were mentioned in the introduction.
\begin{itemize}
	\item[\textbf{(A0)}] \textbf{Basic setting:} The state space is $\tE=U\times U'$ where $U\subset\R^{N+L}$ is $\sigma$-compact and $U'\subset\R^N$ is closed. All of the coefficient functions $f$, $g$, $b$, $\sigma$ are locally Lipschitz continuous and the signal $S_{(\theta,T)}$ is continuous, $T$-periodic with $T\in(0,\infty)$ and depends on some parameter $\theta$ taken from an open set $\Theta\subset\R^D$. 
\end{itemize}
Throughout this article, (A0) will be a tacit standing assumption.

Using the notation $\Phi_t=(X_t,Y_t,Z_t)$ for all $t\in[0,\infty)$ and incorporating the parameters, we rewrite the equation \eqref{eq:SDE} as
\begin{equation}\label{eq:SDEpar}
d\Phi_t=B_{(\theta,T)}(t,\Phi_t)dt+\Sigma(\Phi_t)dW_t,
\end{equation}
where
\[
B_{(\theta,T)}\colon [0,\infty)\times \tE \to \R^{N+L+N}, \quad (t,x,y,z)\mapsto \begin{pmatrix}
f(x,y) +S_{(\theta,T)}(t)+b(z)\\
g(x,y) \\
S_{(\theta,T)}(t)+b(z)
\end{pmatrix},
\]
for each $(\theta,T)\in\Theta\times(0,\infty)$, while
\[
\Sigma \colon \tE \to \R^{(N+L+N)\times M}, \quad (x,y,z)\mapsto \begin{pmatrix}
\sigma(z)\\
0_{L\times M} \\
\sigma(z)
\end{pmatrix}.
\]
We fix some probability space $(\Omega,\cF,\P)$ and we consider the following assumptions about the SDE \eqref{eq:SDEpar}:
\begin{itemize}
	\item[\textbf{(A1)}] \textbf{Unique solvability:} For all $(\theta,T)\in\Theta\times(0,\infty)$ and all deterministic starting points $\Phi_0 \in \tE$, the SDE \eqref{eq:SDEpar} has a unique strong solution $\Phi^{(\theta,T)}=\big(X^{(\theta,T)},Y^{(\theta,T)},Z^{(\theta,T)}\big)\colon[0,\infty)\to\tE$ under $\P$.
	\item[\textbf{(A2)}]  \textbf{Bounded diffusion matrix:} The mapping $\sigma\sigma^\top\colon U'\to\R^{N\times N}$ is uniformly bounded away from $0$ and from $\infty$ in the sense that there are $\sigma_0,\sigma_\infty\in(0,\infty)$ such that
	\[
	\sigma_0 \abs{x}^2 \le x^\top \left(\sigma\sigma^\top(z)\right) x \le \sigma_\infty \abs{x}^2 \quad \text{for all $x \in \R^N$ and $z\in U'$}.
	\]
	\item[\textbf{(A3)}] \textbf{Transition densities for the external variable:} For all $(\theta,T)\in\Theta\times(0,\infty)$ and $t\ge s\ge0$, there is a measurable function
	$p^{(\theta,T)}_{s,t}\colon U' \times U'\to[0,\infty)$ such that
	\[
	\P\left(Z^{(\theta,T)}_t\in B\,\middle|\, Z^{(\theta,T)}_s=z\right)=\int_B p^{(\theta,T)}_{s,t}(z,w)dw \quad \text{for all $z\in U'$ and measurable sets $B\subset U'$.}
	\]
	\item[\textbf{(A4)}] \textbf{Periodic recurrence of the external variable:} For all $(\theta,T)\in\Theta\times(0,\infty)$ the grid chain $\big(Z^{(\theta,T)}_{nT}\big)_{n\in\N_0}$ is positive Harris recurrent.
\end{itemize}

\begin{remark} 1.) As we know from Linear Algebra, (A2) also yields that the inverse $\left(\sigma\sigma^\top(z)\right)^{-1}$ exists for all $z \in U'$, is symmetric and positive definite (and hence possesses a square root $\left(\sigma\sigma^\top(z)\right)^{-1/2}\in\R^{N\times N}$), and we have 
	\begin{equation}\label{eq:ellipticupperbound}
	\sigma_\infty^{-1} \abs{x}^2 \le x^\top \left(\sigma\sigma^\top(z)\right)^{-1} x \le \sigma_0^{-1} \abs{x}^2 \quad \text{for all $x \in \R^N$}.
	\end{equation} 
	
	2.) Note that $\sigma^\top\left(\sigma\sigma^\top\right)^{-1}(z) \in \R^{M\times N}$ is a right inverse of $\sigma(z)$. Thus, the linear mapping $\sigma(z)\colon \R^M\to\R^N$ is surjective and hence $M\ge N$. In this sense, (A2) is a non-degeneracy condition on the external equation for $Z$. It is also "almost sufficient" for (A3) (it is sufficient e.g. in the case that $b$ and $\sigma$ are smooth with bounded derivatives of any order, compare \cite{HairerMalli}).
	
	3.) Together with (A3), the recurrence assumption (A4) allows us to make use of certain variants of classical Limit Theorems (see \cite{HK2}, \cite{HK3}) which we will need for Lemma \ref{lem:LAN} below. Note that (A4) is weaker than the assertion that the entire process $\Phi^{(\theta,T)}$ is positive Harris-recurrent (compare \cite{ICHHarris}).
\end{remark}

Let $(\theta,T)\in\Theta\times(0,\infty)$. We define the probability measure
\[
\P^{(\theta,T)}:=\cL\left([0,\infty)\ni t\mapsto \Phi^{(\theta,T)}_t\,\middle| \,\P\,\right)
\]
on $\cB\big(C([0,\infty);\tE)\big)$ such that for the canonical process $\pi=(\pi_t)_{t\in[0,\infty)}$ on $C([0,\infty);\tE)$ we have
\[
\cL\left(\pi\,\middle|\,\P^{(\theta,T)}\right)=\cL\left(\Phi^{(\theta,T)}\,\middle|\,\P\right).
\]
Observing the process continuously then means working with the filtration given by
\[
\cF_t:=\bigcap_{r\in(t,\infty)} \sigma(\pi_s \,|\,  s\in[0,r]) \subset \cB\big(C([0,\infty);\tE)\big) \quad \text{for all $t\in[0,\infty)$}
\]
and gives rise to the sequence of statistical experiments defined by
\[
\cE_{(X,Y,Z)}:=\left(C([0,\infty);\tE), \cF_n, \left\{ \P^{(\theta,T)}|_{\cF_n} \, \middle| \, (\theta,T) \in \Theta\times(0,\infty) \right\}\right)_{n \in \N}.
\]
As is proved in Section \ref{sect:proofs}, for all $(\tilde\theta,\tilde T)\in\Theta\times(0,\infty)$ the corresponding log-likelihood ratios are given by
\begin{align}\label{eq:loglikelihoodX}
\begin{split}
\log\frac{d\P^{(\tilde\theta,\tilde T)}|_{\cF_t}}{d\P^{(\theta,T)}|_{\cF_t}}
=&\int_0^t \left((\sigma\sigma^\top(\pi^Z_s))^{-1/2}\big(S_{(\tilde\theta,\tilde T)}-S_{(\theta,T)}\big)(s)\right)^\top dB^{(\theta,T)}_s \\ 
&- \frac{1}{2}\int_0^t \big(S_{(\tilde\theta,\tilde T)}-S_{(\theta,T)}\big)^\top(s) \left(\sigma\sigma^\top(\pi^Z_s)\right)^{-1} \big(S_{(\tilde\theta,\tilde T)}-S_{(\theta,T)}\big)(s)ds,
\end{split}
\end{align}
where $B^{(\theta,T)}$ is a Brownian Motion and $\pi^Z=(\pi^{(N+L+1)},\ldots,\pi^{(N+L+N)})$. Examining its structure suggests that in order to find a suitable quadratic expansion for LAN we have to impose appropriate smoothness conditions on the signal with respect to the parameters. The following set of conditions (S1) - (S5) turns out to be sufficient:
\begin{enumerate}
	\item[\textbf{(S1)}] \textbf{Basic regularity:} For each $\theta \in \Theta$ we have a 1-periodic function 
	\[
	S_\theta=\begin{pmatrix}S_\theta^{(1)} \\ \vdots \\ S_\theta^{(N)}\end{pmatrix} \in C^2\big([0,\infty);\R^N\big)
	\]
	such that
	\[
	S_\cdot(s) \in C^1\big(\Theta;\R^N\big) \quad \text{for every $s \in [0,\infty)$}
	\]
	and
	\[
	\d_{\theta_i} S_\theta(\cdot) \in \L^2_{\mathrm{loc}}\big([0,\infty);\R^N\big) \quad \text{for every $\theta\in\Theta$ and $i\in\{1,\ldots,D\}$.}
	\]
	\item[\textbf{(S2)}] \textbf{$\L^2_{\mathrm{loc}}$-differentiability with respect to $(\theta,T)$:} The mapping
	\begin{align*}
	S \colon \Theta \times (0,\infty) &\to \L^2_{\mathrm{loc}}\big([0,\infty);\R^N\big), \\
	(\theta,T) \quad&\mapsto S_{(\theta,T)}:=S_\theta\left(\frac{\cdot}{T}\right),
	\end{align*}
	is $\L^2_{\mathrm{loc}}$-differentiable with the derivative
	\begin{align*}
	\dot{S} \colon \Theta \times (0,\infty) &\to \L^2_{\mathrm{loc}}\big([0,\infty);\R^{N\times(D+1)}\big), \\ (\theta,T) \quad &\mapsto \dot{S}_{(\theta,T)}:=\begin{pmatrix}\d_{\theta_1}S^{(1)}_{(\theta,T)} & \cdots & \d_{\theta_D}S^{(1)}_{(\theta,T)} & \d_T S^{(1)}_{(\theta,T)} \\ \vdots&\ddots&\vdots&\vdots \\\d_{\theta_1}S^{(N)}_{(\theta,T)} & \cdots & \d_{\theta_D}S^{(N)}_{(\theta,T)} & \d_T S^{(N)}_{(\theta,T)} \end{pmatrix},
	\end{align*}
	in the sense that for every $t\in(0,\infty)$ and $(\theta,T) \in \Theta \times (0,\infty)$ we have\footnote{In the context of vector operations, we often write $(\theta,T)$ instead of the formally correct but awkward $(\theta^\top, T)^\top$.}
	\[
	\int_0^t \abs{\frac{S_{(\tilde\theta,\tilde T)}(s)-S_{(\theta,T)}(s)- \dot{S}_{(\theta,T)}(s)\big((\tilde\theta,\tilde T)-(\theta,T)\big)}{\abs{(\tilde\theta,\tilde T)-(\theta, T)}}}^2\!\!ds \to 0, \text{ as } (\tilde\theta,\tilde T) \to (\theta,T).
	\]
	\item[\textbf{(S3)}] \textbf{$\L^2_{\mathrm{loc}}$-continuity of the $(\theta,T)$-derivative:} The mapping $\dot{S}$ is $\L^2_{\mathrm{loc}}$-continuous in the sense that for all $t\in(0,\infty)$ and $(\theta,T) \in \Theta \times (0,\infty)$ we have
	\[
	\int_0^t \abs{\dot{S}_{(\tilde\theta,\tilde T)}(s)-\dot{S}_{(\theta,T)}(s)}^2ds \to 0, \text{ as } (\tilde\theta,\tilde T) \to (\theta,T),
	\]
	where the notation $\abs{\,\cdot\,}$ is used for the Frobenius norm of a matrix.
	\item[\textbf{(S4)}] \textbf{$\L^2_{\mathrm{loc}}$-H{\"o}lder condition with respect to $T$ for the $\theta$-derivative:} For any fixed $\theta\in\Theta$ the mapping
	\begin{align*}
	(0,\infty) \ni T \mapsto D_\theta S_{(\theta,T)}:=\begin{pmatrix}\d_{\theta_1}S^{(1)}_{(\theta,T)} & \cdots & \d_{\theta_D}S^{(1)}_{(\theta,T)} \\ \vdots&\ddots&\vdots \\\d_{\theta_1}S^{(N)}_{(\theta,T)} & \cdots & \d_{\theta_D}S^{(N)}_{(\theta,T)} \end{pmatrix} \in \L^2_{\mathrm{loc}}\big([0,\infty);\R^{N\times D}\big)
	\end{align*}
	satisfies the following local H{\"o}lder condition: For each $T\in(0,\infty)$ there are
	\[
	\alpha \in (0,2] \quad \text{and} \quad \beta\in[0,1+3\alpha/2)
	\]
	such that for suitable $\eps>0$ and $t_0\in[0,\infty)$ we have
	\[
	\int_{t_0}^t \abs{D_\theta S_{(\theta,\tilde T)}(s) - D_\theta S_{(\theta,T)}(s)}^2 ds \le Ct^\beta\abs{\tilde T-T}^\alpha
	\]
	for all $t>t_0$, $\tilde T \in (T-\eps,T+\eps)$, and for some constant $C\in(0,\infty)$ that does not depend on $\tilde T$ or $t$.
	\item[\textbf{(S5)}] \textbf{Linearly independent derivatives:} For all $\theta\in\Theta$, the functions $\d_{\theta_1}S_\theta, \ldots, \d_{\theta_D}S_\theta, S'_\theta$ are linearly independent.
\end{enumerate}

\begin{remark}\label{rem:signal}
	1.) If (S1) holds and $\dot S_{(\theta,T)}(s)$ is continuous (and thus also locally bounded) with respect to $\theta$, $T$, and $s$, (S2) and (S3) follow by dominated convergence. Note that in general, (S1) does not require that for example $\d_{\theta_1}S_{(\theta,T)}(s)$ is continuous (or even locally bounded) in $T$ or $s$.
	
	2.) Suppose that (S1) holds and that for every $\theta\in\Theta$ and $t\in(0,\infty)$ there are $\delta=\delta(\theta) \in (0,1]$ and $C(\theta,t) \le \cst t^\zeta$ with $\zeta \in [0,\delta/2)$ such that the mapping
	\[
	[0,\infty) \ni s \mapsto D_\theta S_\theta(s):= \begin{pmatrix}\d_{\theta_1}S^{(1)}_\theta(s) & \cdots & \d_{\theta_D}S^{(1)}_\theta(s) \\ \vdots&\ddots&\vdots \\\d_{\theta_1}S^{(N)}_\theta(s) & \cdots & \d_{\theta_D}S^{(N)}_\theta(s) \\ \end{pmatrix} \in \R^{N\times d}
	\]
	is H{\"o}lder-$\delta$-continuous on $[0,t]$ with H{\"o}lder-constant $C(\theta,t)$. If $T\in(0,\infty)$, we get that for sufficiently small $\eps>0$ and for all $\tilde T \in (T-\eps,T+\eps)$
	\begin{align*}
		\int_0^{t} \abs{D_\theta S_{(\theta,\tilde T)}(s)-D_\theta S_{(\theta,T)}(s)}^2ds & = \int_0^{t} \abs{D_\theta S_\theta\left(\frac{s}{\tilde T}\right)-D_\theta S_\theta\left(\frac{s}{T}\right)}^2ds \\
		&\le \sup_{T'\in(T-\eps,T+\eps)}C\left(\theta,\frac{t}{T'}\right)^2 \int_0^{t} \abs{\frac{s}{\tilde T}-\frac{s}{T}}^{2\delta}ds \\
		& \le \cst \left(\frac{t}{T-\eps}\right)^{2\zeta} \left(\frac{\abs{\tilde T-T}}{(T-\eps)^2}\right)^{2\delta} \int_0^ts^{2\delta}ds \\
		& \le \cst t^{2\zeta+2\delta+1}\abs{\tilde T-T}^{2\delta}.	
	\end{align*}
	Setting $\alpha:=2\delta$, we can choose
	\[
	\beta:=2(\delta+\zeta)+1 < 2\left(\delta+\frac{\delta}{2}\right)+1= 1+3\alpha/2,
	\]
	and hence the H{\"o}lder condition (S4) is fulfilled.
	
	3.) As a consequence of the two preceding observations, all of the hypotheses (S1) - (S4) are fulfilled if the mapping $\Theta\times[0,\infty) \ni (\theta,s) \mapsto S_\theta(s)$ is in $C^2_b\big(\Theta\times[0,\infty);\R^N\big)$ and $1$-periodic with respect to $s$. Existence and boundedness of $\d_sD_\theta S_\theta(s)$ ensure that we can choose $\delta=1$ and $\zeta=0$ above.
	
	4.) Note that the choice of the matrix norm in (S3) and (S4) is of course arbitrary. We decided to go with the Frobenius norm, because it is commonly used and it is convenient to handle in our calculations.
\end{remark}

The main result is the following one. For a detailed explanation and proof, as well as an explicit introduction of the Fisher Information, we refer to Section \ref{sect:proofs}.

\begin{thm}[Local Asymptotic Normality for $\cE_{(X,Y,Z)}$]\label{thm:LANX}
	Grant all of the hypotheses (A1) - (A4) and (S1) - (S5) and fix $(\theta,T) \in \Theta\times(0,\infty)$. Set
	\[
	\delta_n := \begin{pmatrix}
	n^{-1/2} & 0 & \cdots & 0 \\
	0 & \ddots &\ddots & \vdots \\
	\vdots &\ddots&n^{-1/2}&0 \\
	0 &\cdots&0& n^{-3/2} \\
	\end{pmatrix}
	\in \R^{(D+1)\times(D+1)} \quad \text{for all $n\in\N$,} 
	\]
	and fix any bounded sequence $(h_n)_{n \in \N} \subset \R^{D+1}$. Then $\P^{(\theta,T)}$-almost surely we have
	\begin{equation}
	\log\frac{d\P^{(\theta,T)+\delta_n h_n}|_{\cF_n}}{d\P^{(\theta,T)}|_{\cF_n}}=h_n^\top \cS^{(\theta,T)}_n-\frac12 h_n^\top \cI_{(\theta,T)}h_n + o_{\P^{(\theta,T)}}(1), \quad \text{as $n\to\infty$,}
	\end{equation}
	with Fisher Information $\cI_{(\theta,T)}=\cI_{(\theta,T)}(1)$ as introduced in \eqref{eq:matrix} and score
	\begin{equation*}
		\cS^{(\theta,T)}_n=\delta_n \int_0^n \left((\sigma\sigma^\top)^{-1/2}(\pi^Z_s) \dot{S}_{(\theta,T)}(s)\right)^\top d B^{(\theta,T)}_s \quad \text{for all $n\in\N$} 
	\end{equation*}
	such that weak convergence
	\begin{equation*}
		\cL\left(\cS^{(\theta,T)}_n\middle|\P^{(\theta,T)}\right) \xrightarrow{n\to\infty} \cN\left(0,\cI_{(\theta,T)}\right)
	\end{equation*}
	holds.
\end{thm}

\begin{proof}[Proof of Theorem \ref{thm:LANX}]
	The claim follows immediately from Theorem \ref{thm:LAN} and (the proof of) Proposition \ref{prop:likelihoods}. In particular, the assumptions (A2) and (S5) can in fact be replaced by the slightly weaker but more technical conditions (A2') and (S5') which are introduced in Section \ref{sect:proofs} below and are discussed in Remark \ref{rem:Fischer}.
\end{proof}

Note that other than the basic existence and uniqueness assumption (A1), the conditions for Theorem \ref{thm:LANX} incorporate only the external variable and the deterministic signal. Before we proceed to the proof section, we would like to collect some comments on relevant examples in which these conditions are fulfilled.

\begin{example}\label{ex:OU}
A simple yet important example for the external variable is the multidimensional Ornstein-Uhlenbeck process with time-dependent mean reversion level $S_{(\theta,T)}$. This process corresponds to \eqref{eq:ext} with $b(z)=-\beta z$ for all $z\in U'=\R^N$ with some positive definite $\beta\in\R^{N\times N}$ and a constant volatility $\sigma\in\R^{N\times M}$ such that $\sigma\sigma^\top\in\R^{N\times N}$ is positive definite. Assumption (A2) is then trivially fulfilled, and in complete analogy to the case $M=N=1$ (see \cite[Example 2.3]{HK2}), one can calculate explicitly its transition densities, yielding (A3). These can then be used to apply Theorem 3.2 and Theorem 4.6 (with $f\equiv 1$ and $V(z)=\abs{z}^2$) from \cite{MeynTweedie} in order to check (A4).
\end{example}

\begin{example}\label{ex:signal0}
	1.) Let $S_\theta(s)=F(\theta,\phi(s))$, where $\phi \in C^2\big([0,\infty);\R^K\big)$ is $1$-periodic and 
	\[
	F \colon \Theta\times\R^K \ni (\theta,\xi)=(\theta_1,\ldots,\theta_D,\xi_1,\ldots,\xi_K) \mapsto F(\theta,\xi)=\begin{pmatrix}F_1(\theta,\xi) \\ \vdots \\ F_N(\theta,\xi) \end{pmatrix} \in \R^N 
	\] 
	is continuously differentiable with respect to $\theta\in\Theta$ and twice continuously differentiable with respect to $\xi\in\R^K$. Clearly, the property (S1) holds, and since $\dot S_{(\theta,T)}(s)$ is given by
	\[
	\begin{pmatrix}(\d_{\theta_1} F_1)(\theta,\phi(\frac{s}{T})) & \cdots & (\d_{\theta_D} F_1)(\theta,\phi(\frac{s}{T})) & -sT^{-2} (\nabla_\xi F_1)(\theta,\phi(\frac{s}{T}))^\top \phi'(\frac{s}{T})\\ \vdots&\ddots&\vdots&\vdots \\ (\d_{\theta_1} F_N)(\theta,\phi(\frac{s}{T})) & \cdots & (\d_{\theta_D} F_N)(\theta,\phi(\frac{s}{T})) & -sT^{-2} (\nabla_\xi F_N)(\theta,\phi(\frac{s}{T}))^\top \phi'(\frac{s}{T})  \end{pmatrix}
	\]
	which is continuous with respect to $\theta$, $T$, and $s$, we also have (S2) and (S3). Moreover, we see that the H{\"o}lder property from part 2.)\ of Remark \ref{rem:signal} is fulfilled if it is fulfilled by the mapping
	\[
	\R^K \ni \xi \mapsto \begin{pmatrix}(\d_{\theta_1} F_1)(\theta,\xi) & \cdots & (\d_{\theta_D} F_1)(\theta,\xi)\\ \vdots&\ddots&\vdots \\ (\d_{\theta_1} F_N)(\theta,\xi) & \cdots & (\d_{\theta_D} F_N)(\theta,\xi) \end{pmatrix}.
	\]
	In that case, all of the hypotheses (S1) - (S4) hold.
	
	2.) If the signal has a product structure $S_\theta(s)=D(\theta)\phi(s)$ with $\phi \in C^2\big([0,\infty);\R^K\big)$ $1$-periodic and $G \in C^1\big(\Theta;\R^{N\times K}\big)$, we can treat it as a special case of the preceding example. As for all $s, \tilde s \in [0,\infty)$ we have
	\begin{align*}
	\abs{D_\theta S_\theta(s)-D_\theta S_\theta(\tilde s)}^2 =& \sum_{n=1}^N\sum_{d=1}^D \left(\sum_{k=1}^K (\d_{\theta_d} G_{n,k})(\theta) \big(\phi_k(s)-\phi_k(\tilde s)\big) \right)^2 \\ 
	\le & \left(\sum_{n=1}^N\sum_{d=1}^D \sum_{k=1}^K (\d_{\theta_d} G_{n,k})^2(\theta)\right) \abs{\phi(s)-\phi(\tilde s)}^2 \\
	\le &\left(\sum_{n=1}^N\sum_{d=1}^D \sum_{k=1}^K (\d_{\theta_d} G_{n,k})^2(\theta)\right) \norm{\phi'}_\infty^2 \abs{s-\tilde s}^2 ,
	\end{align*}
	no further conditions are needed to ensure the H{\"o}lder property from part 2.)\ of Remark \ref{rem:signal} to hold with $\delta=1$ and $\zeta=0$.
	
	3.) In particular, the example above secures that (S1) - (S4) are fulfilled for signals of the form
	\begin{equation}\label{eq:signalsin}
	S_\theta(s)=\sum_{k=1}^K \big( \sin(2k\pi s) G_k(\theta)+ \cos(2k\pi s)H_k(\theta) \big) \quad \text{for all $s\in[0,\infty)$}
	\end{equation}
	with $K \in \N$ and $G_k,H_k \in C^1\big(\Theta;\R^N\big)$ for all $k \in\{1,\ldots,K\}$.
	
	4.) Taking $K=D$, $N=1$ and $G_k(\theta)=\theta_k$, $H_k(\theta)=0$ for all $\theta\in\Theta$ and $k \in\{1,\ldots,K\}$, the signal from \eqref{eq:signalsin} clearly also satisfies (S5), als long as $0\notin\Theta$.
\end{example}

\section{Proofs and supplementary results}\label{sect:proofs}

\subsection{Observing $(X,Y,Z)$, $X$, or $Z$}\label{sect:observe}

We start this section with a fundamental observation: If the starting point is known, observing only the adjustable variable $X$ is actually no restriction, since we can successively reconstruct the remaining variables $Y$ and $Z$. Let us explain this step for step in the following remark.

\begin{remark}\label{rem:observe}
	Assume that the starting point $(X_0,Y_0,Z_0)\in\tE$ is known. Fix a finite time horizon $t_0\in(0,\infty)$ and assume that the trajectory $(X_t)_{t\in[0,t_0]}$ has been observed and is thus also known. Then the function $(t,y) \mapsto g(X_t,y)$ is completely known, and given the structure of the internal equation in \eqref{eq:SDE}, the trajectory $(Y_t)_{t\in[0,t_0]}$ is now given as the solution to the ordinary differential equation
	\begin{align*}
		dY_t=g(X_t,Y_t)dt \quad \text{for all $t\in[0,t_0]$.}
	\end{align*}
	Now we know both $(X_t)_{t\in[0,t_0]}$ and $(Y_t)_{t\in[0,t_0]}$, and by rearranging the first line of \eqref{eq:SDE}, this information allows us to calculate
	\[
	Z_t=Z_0+X_t-X_0-\int_0^t f(X_s,Y_s)ds \quad \text{for all $t\in[0,t_0]$.}
	\]
	All in all, we have reconstructed every component of $(X_t,Y_t,Z_t)_{t\in[0,t_0]}$ just from $(X_t)_{t\in[0,t_0]}$ and the starting point $(X_0,Y_0,Z_0)$.
\end{remark}

Remark \ref{rem:observe} is the legitimation for us to work with the idealized assumption that we can in fact observe the entire process $(X,Y,Z)$ even in situations where realistically one could only observe the adjustable variable $X$. Next, we will describe the corresponding statistical experiment.

In order to make Proposition \ref{prop:likelihoods} more apprehensible, we will do this very carefully and with much attention to measure-theoretic subtleties. A look at \eqref{eq:SDEpar} reveals that the drift coefficient depends on the parameter $(\theta,T)\in\Theta\times(0,\infty)$, while the volatility does not. Hence, we can use \cite[Theorem 6.10]{HoBo}\footnote{Note that we do not assume -- as in this Theorem -- that $B$ and $\Sigma$ are defined on the entire euclidean space and are globally Lipschitz continuous. By our assumptions, $\tE=U\times U'$ is $\sigma$-compact and hence we can find a sequence $(K_n)_{n\in\N}$ of compact sets increasing to $\tE$. Using Kirszbraun's Theorem (\cite[Hauptsatz I]{Kirsz}), the restriction to each $K_n$ of $B(t,\cdot)$ and $\Sigma$ can be extended to globally Lipschitz continuous functions on $\R^{N+L+N}$ (which also satisfy a linear growth condition). Hence, the proof of \cite[Theorem 6.10]{HoBo} needs only a slight adjustment to work in our case: Using the notation from there, the stopping time $\rho_n$ has to be replaced by $\rho_n \wedge \inf\{t>0\,|\, \eta_t \notin K_n\}$ and in equation ($\mathrm{II}^{(n)}$) and thereafter the coefficients $b$, $\sigma$ and $c$ have to be altered in analogy to $\gamma$. The rest of the proof then needs no further changes.} in order to determine the log-likelihood ratios. Let $(t,x,y,z)\in[0,\infty)\times\tE$. Comparing the drift coefficients of \eqref{eq:SDEpar} with different parameters $(\tilde\theta,\tilde T),(\theta,T)\in\Theta\times(0,\infty)$, we see that
\begin{align*}
	\big(B_{(\tilde\theta,\tilde T)}-B_{(\theta,T)}\big)(t,x,y,z)= \Sigma\Sigma^\top(x,y,z) \Gamma(t,x,y,z),
\end{align*}
where
\[
\Gamma(t,x,y,z):=\begin{pmatrix}  0 \\ \big(\sigma\sigma^\top\big)^{-1}(z)\big(S_{(\tilde\theta,\tilde T)}-S_{(\theta,T)}\big)(t) \end{pmatrix} \in \R^{N+L+N}.
\]
Thanks to (A2) and \eqref{eq:ellipticupperbound}, 
\begin{equation}\label{eq:A2ref}
	\int_0^t \big(\Gamma^\top \Sigma\Sigma^\top \Gamma\big)(s,\pi_s) ds 
	\le \sigma_0^{-1}\int_0^t \abs{S_{(\tilde\theta,\tilde T)}(s)-S_{(\theta,T)}(s)}^2ds <\infty,
\end{equation}
because the signals are continuous. Thence, both conditions (+) and (++) of \cite[Theorem 6.10]{HoBo} are fulfilled. Writing $m^{\Phi,(\theta,T)}$ for the local martingale part of $\pi$ under $\P^{(\theta,T)}$, we can conclude that
\[
\log\frac{d\P^{(\tilde\theta,\tilde T)}|_{\cF_t}}{d\P^{(\theta,T)}|_{\cF_t}}
=\int_0^t \Gamma(s,\pi_s)^\top dm^{\Phi,(\theta,T)}_s - \frac{1}{2}\int_0^t \left(\Gamma^\top \Sigma\Sigma^\top \Gamma\right)(s,\pi_s) ds.
\]
Setting $\pi^Z:=\left(\pi^{(N+L+1)},\ldots,\pi^{(N+L+N)}\right)^\top$ and writing $m^{Z,(\theta,T)}$ for its local martingale part under $\P^{(\theta,T)}$, the expression for the log-likelihood ratio can be rewritten as
\begin{align*}
	\int_0^t \Big( \left(\sigma\sigma^\top(\pi^Z_s)\right)^{-1} & \big(S_{(\tilde\theta,\tilde T)}-S_{(\theta,T)}\big)(s) \Big)^\top dm^{Z,(\theta,T)}_s \\
	&- \frac{1}{2}\int_0^t \big(S_{(\tilde\theta,\tilde T)}-S_{(\theta,T)}\big)^\top(s) \left(\sigma\sigma^\top(\pi^Z_s)\right)^{-1} \big(S_{(\tilde\theta,\tilde T)}-S_{(\theta,T)}\big)(s)ds.
\end{align*}
In order to eliminate the rather unintuitive integral with respect to $m^{Z,(\theta,T)}$, we introduce the local $\big(\P^{(\theta,T)},(\cF_t)_{t\in[0,\infty)}\big)$-martingale $B^{(\theta,T)}:=\big(B^{(\theta,T)}_t\big)_{t\in[0,\infty)}$ given by
\begin{equation}\label{eq:BMX}
B^{(\theta,T)}_t=\int_0^t (\sigma\sigma^\top)^{-1/2}(\pi^Z_s) dm^{Z,(\theta,T)}_s \quad \text{for all $t\in[0,\infty)$.}
\end{equation}
Its quadratic variation process is
\begin{align*}
	\left\langle \int_0^\cdot (\sigma\sigma^\top)^{-1/2}(\pi^Z_s) dm^{Z,(\theta,T)}_s \right\rangle_t
	= \int_0^t (\sigma\sigma^\top)^{-1}(\pi^Z_s) d \left(\int_0^s\sigma\sigma^\top(\pi^Z_r)dr\right)
	= t\cdot 1_{N\times N}
\end{align*}
for all $t\in[0,\infty)$, so L{\'e}vy's Characterization Theorem \cite[Theorem II.6.1]{Ikeda} yields that $B^{(\theta,T)}$ is an $N$-dimensional $\big(\P^{(\theta,T)},(\cF_t)_{t\in[0,\infty)}\big)$-Brownian Motion. Incorporating this process, we can write
\begin{align}\label{eq:loglikelihoodX}
	\begin{split}
		\log\frac{d\P^{(\tilde\theta,\tilde T)}|_{\cF_t}}{d\P^{(\theta,T)}|_{\cF_t}}
		=&\int_0^t \left((\sigma\sigma^\top(\pi^Z_s))^{-1/2}\big(S_{(\tilde\theta,\tilde T)}-S_{(\theta,T)}\big)(s)\right)^\top dB^{(\theta,T)}_s \\ 
		&- \frac{1}{2}\int_0^t \big(S_{(\tilde\theta,\tilde T)}-S_{(\theta,T)}\big)^\top(s) \left(\sigma\sigma^\top(\pi^Z_s)\right)^{-1} \big(S_{(\tilde\theta,\tilde T)}-S_{(\theta,T)}\big)(s)ds.
	\end{split}
\end{align}
We note immediately that the only component of $\pi$ that is featured explicitly in this expression is the $\pi^Z$-component. It seems plausible that we should get the same expression for the log-likelihood ratio in an experiment that does not even know that any variables other than $Z$ exist. Let us make this formally rigorous.

Let $\eta=(\eta_t)_{t\in[0,\infty)}$ be the canonical process on $C\big([0,\infty);U'\big)$, and write
\[
\Q^{(\theta,T)}:=\cL\left([0,\infty)\ni t\mapsto Z^{(\theta,T)}_t\,\middle| \,\P\,\right)
\]
for the law on $\cB\big(C\big([0,\infty);U')\big)$ of the unique strong solution $Z^{(\theta,T)}$ on $(\Omega,\cF)$ under $\P$ of
\begin{equation}\label{eq:ext}
dZ_t = [S_{(\theta,T)}(t)+b(Z_t)]dt + \sigma(Z_t)dW_t,
\end{equation}
when issued from $Z_0\in\R^N$ with the parameter $(\theta,T)\in \Theta\times(0,\infty)$. For any $t\in[0,\infty)$ let
\[
\cG_t:=\bigcap_{r\in(t,\infty)} \sigma(\eta_s \,|\, s\in[0,r]) \subset \cB\big(C\big([0,\infty);U'\big)\big)
\]
and consider the sequence of experiments given by
\begin{equation}\label{eq:cE_Z}
\cE_Z:=\left(C\big([0,\infty);U'\big), \cG_n, \left\{ \Q^{(\theta,T)}|_{\cG_n} \, \middle| \, (\theta,T) \in \Theta\times(0,\infty) \right\}\right)_{n \in \N}.
\end{equation}
Using the same arguments as above and writing $\tilde m^{Z,(\theta,T)}$ for the local martingale part of $\eta$ under $\Q^{(\theta,T)}$, we can again use \cite[Theorem 6.10]{HoBo} and conclude
\begin{align}\label{eq:loglikelihoodZ}
	\begin{split}
		\log\frac{d\Q^{(\tilde\theta,\tilde T)}|_{\cG_t}}{d\Q^{(\theta,T)}|_{\cG_t}}
		=&\int_0^t \left((\sigma\sigma^\top)^{-1/2}(\eta_s)\big(S_{(\tilde\theta,\tilde T)}-S_{(\theta,T)}\big)(s)\right)^\top d\tilde B^{(\theta,T)}_s \\ 
		&- \frac{1}{2}\int_0^t \big(S_{(\tilde\theta,\tilde T)}-S_{(\theta,T)}\big)^\top(s) \left(\sigma\sigma^\top(\eta_s)\right)^{-1} \big(S_{(\tilde\theta,\tilde T)}-S_{(\theta,T)}\big)(s)ds,
	\end{split}
\end{align}
where the process $\tilde B^{(\theta,T)}:=\big(\tilde B^{(\theta,T)}_t\big)_{t\in[0,\infty)}$ given by
\begin{equation}\label{eq:BMZ}
\tilde B^{(\theta,T)}_t=\int_0^t (\sigma\sigma^\top)^{-1/2}(\eta_s) d\tilde m^{Z,(\theta,T)}_s \quad \text{for all $t\in[0,\infty)$}
\end{equation}
is again an $N$-dimensional $\big(\Q^{(\theta,T)},(\cG_t)_{t\in[0,\infty)}\big)$-Brownian Motion.

We now have calculated the log-likelihood ratios for both $\cE_{(X,Y,Z)}$ and $\cE_Z$. Comparing them leads to the following result. 

\begin{prop}\label{prop:likelihoods}
	Grant assumptions (A1) and (A2). The sequences $\cE_{(X,Y,Z)}$ and $\cE_Z$ corresponding to continuous observation of $(X,Y,Z)$ or $Z$ respectively, with the same deterministic starting point $(X_0,Y_0,Z_0)\in\tE$, are statistically equivalent in the sense that
	\begin{equation}\label{eq:likelihoods}
	\cL\left( \left(\log\frac{d\Q^{(\tilde\theta,\tilde T)}|_{\cG_t}}{d\Q^{(\theta,T)}|_{\cG_t}}\right)_{t\in[0,\infty)} \;\middle|\; \Q^{(\theta,T)} \right) = \cL\left( \left(\log\frac{d\P^{(\tilde\theta,\tilde T)}|_{\cF_t}}{d\P^{(\theta,T)}|_{\cF_t}}\right)_{t\in[0,\infty)} \;\middle|\; \P^{(\theta,T)} \right)
	\end{equation}
	for all $(\theta,T), (\tilde\theta,\tilde T)\in\Theta\times(0,\infty)$. In particular, we have LAN for $\cE_{(X,Y,Z)}$ if and only if we have it for $\cE_Z$ with the same local scale, the same Fisher Information and an identically distributed Score.
\end{prop}


\begin{proof}
	Due to the definition of $\P^{(\theta,T)}$ and $\Q^{(\theta,T)}$, we have
	\[
	\cL\big(\eta\,\big|\,\Q^{(\theta,T)}\big)=\cL\big(\pi^Z\,\big|\,\P^{(\theta,T)}\big),
	\]
	and in view of \eqref{eq:BMX}, \eqref{eq:loglikelihoodX}, \eqref{eq:loglikelihoodZ}, and \eqref{eq:BMZ}, this implies \eqref{eq:likelihoods} from which the second statement of this Proposition follows immediately.
\end{proof}

In view of Theorem \ref{thm:LANX}, Proposition \ref{prop:likelihoods} is the justification for us to restrict ourselves to studying the simpler process $Z$ instead of the more complex $(X,Y,Z)$ in the following section.

\subsection{Local Asymptotic Normality for $Z$}\label{sect:LAN}

This section centres around the sequence of statistical experiments defined by $\cE_Z$ in \eqref{eq:cE_Z} which corresponds to continuous observation over growing time intervals of the $N$-dimensional diffusion $Z$ following the parameter-dependent SDE \eqref{eq:ext}. As mentioned in Section \ref{sect:intro}, taking $M=N=1$, $b\equiv0$, and $\sigma\equiv1$ leads to the classical "signal in white noise" model.  For this special case, Ibragimov and Khasminskii proved LAN with rate $n^{-3/2}$ for a smooth signal with known $\theta$ and unknown $T$, and discussed asymptotic efficiency for certain estimators (see \cite[Sections II.7 and III.5]{Ibra}). In \cite{Golubev}, Golubev extended their approach with $\L^2$-methods in order to estimate $T$ at the same rate for unknown shape which in turn was the basis for Castillo, L\'{e}vy-Leduc and Matias for non-parametric estimation of the shape under unknown $T$ (see \cite{CLM}). For our more general diffusion \eqref{eq:ext}, we will stay within the confines of parametric estimation. The main result of this section is LAN for the sequence of experiments $\cE_Z$ with unknown $\theta$ and unknown $T$ (Theorem \ref{thm:LAN}). For $M=N=1$ H{\"o}pfner and Kutoyants had already solved this problem both for known $T$ with unknown $\theta$ (see \cite{HK1}) and for known $\theta$ with unknown $T$ (see \cite{HK3}). A result on LAN jointly in $\theta$ and $T$ was presented in \cite{ICH}, but still only in dimension one. Theorem \ref{thm:LAN} extends all of these results and allows for application to simultaneous estimation of the shape and the periodicity in any dimension.

In the context of this subsection, we replace the assumption (A1) with the following weaker analogue. 
\begin{itemize}
	\item[\textbf{(A1')}] \textbf{Unique solvability:} For all $(\theta,T)\in\Theta\times(0,\infty)$ and all deterministic starting points $Z_0\in U'$, the SDE \eqref{eq:ext} has a unique strong solution $Z^{(\theta,T)}\colon[0,\infty)\to U'$ under $\P$.
\end{itemize}
We also work with the following slight relaxation of (A2).
\begin{itemize}
	\item[\textbf{(A2')}]  \textbf{Uniform ellipticity:} The mapping $\sigma\sigma^\top\colon U'\to\R^{N\times N}$ is uniformly elliptic, i.e.\ there is some $\sigma_0\in(0,\infty)$ such that
	\[
	x^\top \left(\sigma\sigma^\top(z)\right) x \ge \sigma_0 \abs{x}^2 \quad \text{for all $x \in \R^N$ and $z\in U'$}.
	\]
\end{itemize}
Note that so far, the only use of (A2) occured in \eqref{eq:A2ref}, and there (A2') would also suffice. Let us also give an equivalent reformulation of (A4) which incorporates the notation we introduced in the previous section.
\begin{itemize}
	\item[\textbf{(A4)}] \textbf{Periodic recurrence of \eqref{eq:ext}:} For all $(\theta,T)\in\Theta\times(0,\infty)$ the grid chain $\left(\eta_{kT}\right)_{k\in\N_0}$ under $\Q^{(\theta,T)}$ is positive Harris recurrent with invariant probability measure $\mu^{(\theta,T)}$.
\end{itemize}
Periodicity of the signal is the reason why (A4) even makes sense at all: Since $S_{(\theta,T)}$ and therefore the entire drift term of \eqref{eq:ext} is $T$-periodic, the grid chain is a $U'$-valued time-homogeneous discrete-time Markov process. Another important process that is embedded in $\eta$ in a similar way is the $C([0,T];U')$-valued time-homogeneous \emph{path segment chain} $\eta^{\mathbf{ps}}:=\left(\eta^{\mathbf{ps}}_k\right)_{k\in\N_0}$ defined by taking an arbitrary $\eta^{\mathbf{ps}}_0\in C\big([0,T];U'\big)$ with $\eta^{\mathbf{ps}}_0(T)=Z_0$ and then setting
\[
	\eta^{\mathbf{ps}}_k:=\left([0,T]\ni t \mapsto \eta_{(k-1)T+t}\right) \quad \text{for all $k\in\N$.}
\]
As we know from \cite[Theorem 2.1 (a)]{HK2}\footnote{Note that even though this Theorem is only explicitly stated for $\R$-valued processes, the authors remark at the beginning of the section that it remains valid for any polish state space, in particular for the cloed set $U'\subset\R^N$.}, the path segment chain $\eta^{\mathbf{ps}}$ inherits positive Harris recurrence under $\Q^{(\theta,T)}$ from the grid chain and its invariant distribution $m^{(\theta,T)}$ is the unique measure on $\cB\big(C([0,T];U')\big)$ such that for all $l\in\N$, $0=t_0<t_1<\ldots<t_l=T$, and $B_0,\ldots,B_l\in\cB\big(U'\big)$ we have
\begin{align}\label{eq:segmentinvariantmeasure}
	\begin{split}
		m^{(\theta,T)}( \eta_{t_i} \in B_i &\text{ for all $i \in \{0,\ldots,l\}$} ) 
		= \int_{B_0} \mu^{(\theta,T)}(dx_0)\int_{B_1}Q^{(\theta,T)}_{t_0,t_1}(x_0,dx_1)\ldots \int_{B_l}Q^{(\theta,T)}_{t_{l-1},t_l}(x_{l-1},dx_l),
	\end{split}
\end{align}
where $\big(Q^{(\theta,T)}_{s,t}\big)_{t> s\ge 0}$ is the transition semi-group of $\eta$ under $\Q^{(\theta,T)}$.

We will make use of the following strong law of large numbers for the path segment chain which we cite from \cite[Theorem 2.1 (b)]{HK2}.

\begin{prop}\label{prop:SLLN}
	Let (A1'), (A3) and (A4) hold and fix some $(\theta,T)\in\Theta\times(0,\infty)$. Assume that $\left(A_t\right)_{t\in[0,\infty)}$ is a $\big(\Q^{(\theta,T)},(\cG_t)_{t\in[0,\infty)}\big)$-increasing process. If there is a non-negative function $F \in \L^1\big(m^{(\theta,T)}\big)$ such that
	\[
	A_{kT}=\sum_{j=1}^k F\left(\eta^{\mathbf{ps}}_j\right) \quad \text{$\Q^{(\theta,T)}$-almost surely for all $k \in \N$,}
	\]
	then
	\[
	\frac{1}{t}A_t \xrightarrow{t \to \infty} \frac{1}{T} \int_{C([0,T];U')}F(\phi) m^{(\theta,T)}(d\phi) \quad \text{$\Q^{(\theta,T)}$-almost surely.}
	\]
\end{prop}

\begin{proof}
	See Section 2 of \cite{HK2}.
\end{proof}

Proposition \ref{prop:SLLN} is the key to the following Lemma \ref{lem:LAN} which is a slightly modified multi-dimensional version of Lemmas 2.1 and 2.2 from \cite{HK3}.

\begin{lem}\label{lem:LAN}
	Grant assumptions (A1'), (A3) and (A4). Further assume that the measurable mapping $G\colon U'\to\R^{N\times N}$ has values only in the set of symmetric matrices and is uniformly elliptic. We define the mapping
	\begin{align}\label{eq:Q}
		\begin{split}
			\B^{(\theta,T)}_G\colon \big(\L^2\big([0,1];\R^N\big)\big)^2&\;\to\;\hspace{1.1cm}\R,\\
			(u,v)\hspace{1.1cm}&\;\mapsto\; \int_0^1 u(s)^\top \Big(\mu^{(\theta,T)} Q^{(\theta,T)}_{0,sT}(G^{-1})\Big) v(s)ds,
		\end{split}
	\end{align}
	where
	\[
	\mu^{(\theta,T)} Q^{(\theta,T)}_{0,sT}(G^{-1})=\int_{U'}\mu^{(\theta,T)}(dz)\int_{U'}  Q^{(\theta,T)}_{0,sT}(z,d\tilde z) G^{-1}(\tilde z) \in \R^{N\times N}
	\]
	is understood as a matrix-valued integral. Then the following statements are true.
	\begin{enumerate}
		\item[(i)] $\B^{(\theta,T)}_G$ is a non-negative definite and symmetric bilinear form.
		\item[(ii)] If we consider $u,v \in \L^2\big([0,1];\R^N\big)$ as 1-periodic functions on $[0,\infty)$, then for any $k\in\N_0$ we have
		\begin{equation}\label{eq:lemLANconv}
		\frac{k+1}{t^{k+1}}\int_0^t s^k u(s/T)^\top G^{-1}(\eta_s) v(s/T)ds \xrightarrow{t \to \infty} \B^{(\theta,T)}_G[u,v]
		\end{equation}
		$\Q^{(\theta,T)}$-almost surely.
	\end{enumerate}
\end{lem}

\begin{proof}
	For the sake of simplicity and as $(\theta,T)$ is fixed anyway, we drop all corresponding superscripts. First, we check that $\B_G$ is indeed a well-defined mapping with values in $\R$. Let the lower bound for the eigenvalues of $G(\cdot)$ be denoted by $G_0 \in (0,\infty)$. Recall that $G^{-1}(\cdot)$ always exists, is positive definite, and $G_0^{-1}$ is an upper bound for its eigenvalues. Then by linearity and contractivity of the operator $\mu Q_{0,sT}$, we can estimate
	\[
	0\le \B_G[u,u] = \int_0^1 \mu Q_{0,sT}\left( u(s)^\top G^{-1}(\cdot)u(s) \right)ds \le G_0^{-1}\int_0^1 \abs{u(s)}^2ds <\infty.
	\]
	Thanks to the symmetry of $G^{-1}$, we can polarize the integrand and thus the whole expression, which allows us to use the above in order to conclude that
	\[
	\abs{\B_G[u,v]}=\frac12 \abs{ \B_G[u,u]+\B_G[v,v]-\B_G[u+v,u+v]} < \infty,
	\]
	and hence $\B_G$ is well-defined. It is then trivial to see that it is a non-negative definite and symmetric bilinear form, and the proof for (i) is complete.
	
	We note that the left hand side of \eqref{eq:lemLANconv} is bilinear in $u$ and $v$ as well. Thanks to this and (i), the proof of the second statement of the Lemma can be reduced to the case $u=v$, since the general case then follows by polarization.
	
	Let us fix $u \in \L^2\big([0,1];\R^N\big)$ and define the process $A:=(A_t)_{t\in[0,\infty)}$ with
	\[
	A_t:=\int_0^t u(s/T)^\top G^{-1}(\eta_s) u(s/T)ds \quad \text{for all $t\in[0,\infty)$.}
	\]
	Since $G^{-1}(\cdot)$ is positive definite, the integrand is non-negative, and therefore $A$ is an increasing process whose trajectories are obviously continuous. Note that the expression on the left hand side of \eqref{eq:lemLANconv} can be rewritten as
	\[
	\frac{k+1}{t^{k+1}}\int_0^t s^k dA_s.
	\]
	For $k=0$ this is simply $\frac{1}{t}A_t$, which we will handle with the help of Proposition \ref{prop:SLLN}. The general statement then follows from this special case by elementary calculus (compare Lemma 3.17 of \cite{ICHDiss}).	
	
	In order to establish the functional relation between $A$ and $\eta$ that is needed in Proposition \ref{prop:SLLN}, we define the function
	\[
	F \colon C\big([0,T];U'\big) \to [0,\infty), \quad \phi \mapsto \int_0^T  u(s/T)^\top G^{-1}(\phi(s)) u(s/T)ds,
	\]
	which is bounded by $T G_0^{-1}\norm{u}_{\L^2([0,1])}$, and thus it is integrable with respect to the probability measure $m$. Since $u$ is 1-periodic, we see that
	\begin{align*}
		\sum_{j=1}^k F\left(\eta^{\mathbf{ps}}_j\right)&=\sum_{j=1}^k \int_0^T u(s/T)^\top G^{-1}(\eta_{(j-1)T+s}) u(s/T)ds
		=\int_0^{kT} u(s/T)^\top G^{-1}(\eta_s) u(s/T)ds 
		=A_{kT}
	\end{align*}
	for all $k \in \N$, and consequently Proposition \ref{prop:SLLN} allows to deduce $\Q$-almost sure convergence
	\begin{align*}
		\lim_{t \to \infty} \frac{1}{t}A_t 
		&= \frac{1}{T} \int_{C([0,T];U')} \int_0^T  u(s/T)^\top G^{-1}(\phi(s)) u(s/T)ds \,  m(d\phi) \\ 
		&= \frac{1}{T} \int_0^T u(s/T)^\top \left( \int_{C([0,T];U')} G^{-1}(\phi(s))m(d\phi)\right) u(s/T)ds \\
		&= \frac{1}{T} \int_0^T u(s/T)^\top \left(\int_{U'} G^{-1}(x) \mu Q_{0,s}(dx)\right) u(s/T)ds \\
		&= \B_G[u,u],
	\end{align*}
	where the use of Fubini's Theorem in the second step is justified by the non-negativity of the integrand, and the third step makes use of \eqref{eq:segmentinvariantmeasure}. This completes the proof.
\end{proof}

Using the notation from \eqref{eq:Q}, for each $(\theta,T)\in\Theta\times(0,\infty)$ and $t\in[0,\infty)$ we define the symmetric $(D+1)\times(D+1)$-dimensional block matrix
\begin{equation}\label{eq:matrix}
\cI_{(\theta,T)}(t):= \begin{pmatrix}
t \left(\B^{(\theta,T)}_{\sigma\sigma^\top}[ \d_{\theta_i}S_\theta,\d_{\theta_j}S_\theta]\right)_{i,j=1,\ldots,D} & -\frac{t^2}{2T^2} \left(\B^{(\theta,T)}_{\sigma\sigma^\top}[\d_{\theta_i}S_\theta,S_\theta']\right)_{i=1,\ldots,D} \\
\cdots & \frac{t^3}{3T^4} \B^{(\theta,T)}_{\sigma\sigma^\top}[S_\theta',S_\theta']
\end{pmatrix}.
\end{equation}
Its derivative with respect to $t$ is given by
\[
\cI_{(\theta,T)}'(t)=\begin{pmatrix}
\left(\B^{(\theta,T)}_{\sigma\sigma^\top}[ \d_{\theta_i}S_\theta,\d_{\theta_j}S_\theta]\right)_{i,j=1,\ldots,D} & -tT^{-2}\left(\B^{(\theta,T)}_{\sigma\sigma^\top}[\d_{\theta_i}S_\theta,S_\theta']\right)_{i=1,\ldots,D} \\
\cdots & t^2T^{-4}\B^{(\theta,T)}_{\sigma\sigma^\top}[S_\theta',S_\theta']
\end{pmatrix}.
\]
We make the following assumption.
\begin{itemize}
	\item[\textbf{(S5')}] \textbf{Regularity of the signal with respect to $\B^{(\theta,T)}_{\sigma\sigma^\top}$:} For all $(\theta,T)\in\Theta\times(0,\infty)$ and $t\in(0,\infty)$ we have
	\[
	\text{(i) $\cI_{(\theta,T)}(t)$ is invertible,} \qquad\quad \text{(ii) $\cI_{(\theta,T)}'(t)$ is invertible.}
	\]
\end{itemize}
While part (ii) of (S5') is merely needed for technical reasons (as will become clear in the proof of Theorem \ref{thm:LAN} below), part (i) is of more general importance, since $\cI_{(\theta,T)}(1)$ will turn out to be the Fisher Information. We will discuss these conditions in detail in the following remark.

\begin{remark}\label{rem:Fischer}
	1.) Note that $\cI'_{(\theta,T)}(t)$ is the Gramian matrix of $\d_{\theta_1}S_\theta,\ldots,\d_{\theta_D}S_\theta,-tT^{-2}S_\theta'$ with respect to the non-negative definite symmetric bilinear form $\B^{(\theta,T)}_{\sigma\sigma^\top}$. Hence, it is non-negative definite. The same is true for $\cI_{(\theta,T)}(t)$, since it is "almost a Gramian matrix". Indeed, setting
	\[
	u_1:=t^{1/2}\d_{\theta_1}S_\theta, \,\ldots,\, u_D:=t^{1/2}\d_{\theta_D}S_\theta, \,u_{D+1}:= -\frac{t^{3/2}}{2T^2}S'_\theta,
	\]
	we can write
	\[
	\renewcommand{\arraystretch}{1.6}
	\cI_{(\theta,T)}(t) =  \begin{pmatrix}
	\B^{(\theta,T)}_{\sigma\sigma^\top}[u_1,u_1] & \cdots & \cdots & \B^{(\theta,T)}_{\sigma\sigma^\top}[u_1,u_{D+1}] \\
	\vdots &\ddots &  & \vdots\\ 
	\vdots & & \B^{(\theta,T)}_{\sigma\sigma^\top}[u_D,u_D] & \B^{(\theta,T)}_{\sigma\sigma^\top}[u_D,u_{D+1}]\\ 
	\B^{(\theta,T)}_{\sigma\sigma^\top}[u_{D+1},u_1] & \cdots & \B^{(\theta,T)}_{\sigma\sigma^\top}[u_{D+1},u_D] & \frac43 \B^{(\theta,T)}_{\sigma\sigma^\top}[u_{D+1},u_{D+1}]\end{pmatrix},
	\renewcommand{\arraystretch}{1}
	\]
	and we see that for all $x\in\R^{D+1}$
	\begin{align*}
	x^\top \cI_{(\theta,T)}(t) x &= \sum_{i,j=1}^{D+1} x_i \B^{(\theta,T)}_{\sigma\sigma^\top}[u_i,u_j] x_j + \frac13 x^2_{D+1} \B^{(\theta,T)}_{\sigma\sigma^\top}[u_{D+1},u_{D+1}] 
	\ge \B^{(\theta,T)}_{\sigma\sigma^\top}\left[\sum_{i=1}^{D+1}x_iu_i,\sum_{j=1}^{D+1}x_ju_j\right]
	\end{align*}
	which is non-negative.
	
	2.) In particular, 1.) implies that $\cI_{(\theta,T)}(t)$ and $\cI'_{(\theta,T)}(t)$ are invertible if and only if they are positive definite.
	
	3.) If $\B^{(\theta,T)}_{\sigma\sigma^\top}$ is positive definite (and hence an inner product), the same reasoning as in 1.) yields that linear independence of $\d_{\theta_1}S_\theta, \ldots, \d_{\theta_D}S_\theta, S'_\theta$ is equivalent to invertibility of $\cI'_{(\theta,T)}(t)$, and sufficient for invertibility of $\cI_{(\theta,T)}(t)$.
	
	4.) If (A2) holds, for all $u\in\L^2\big([0,1];\R^N\big)$ we can use \eqref{eq:ellipticupperbound} and estimate
	\[
	\B^{(\theta,T)}_{\sigma\sigma^\top}[u,u] = \int_0^1 \mu^{(\theta,T)} Q^{(\theta,T)}_{0,sT}\left( u(s)^\top \big(\sigma\sigma^\top\big)^{-1}(\cdot)u(s) \right)ds \ge \sigma_\infty^{-1} \int_0^1 \abs{u(s)}^2ds,
	\]
	i.e.\ $\B^{(\theta,T)}_{\sigma\sigma^\top}$ is positive definite (in fact even coercive). Thus, (A2) and (S5) together imply (S5').
	
	5.) A very simple and seemingly natural sufficient condition for (S5') is orthogonality of the functions $\d_{\theta_1}S_\theta, \ldots, \d_{\theta_D}S_\theta, S'_\theta$ with respect to $\B^{(\theta,T)}_{\sigma\sigma^\top}$ (without assuming this bilinear form to be positive definite). This is equivalent to both $\cI_{(\theta,T)}(t)$ and $\cI_{(\theta,T)}'(t)$ being diagonal matrices with non-vanishing diagonal entries and as such they are invertible. However, this is not a very likely scenario, since $S_\theta$ has $D$ degrees of freedom, determines the $D$ functions $\d_{\theta_1}S_\theta, \ldots, \d_{\theta_D}S_\theta$, and then $S'_\theta$ -- while adding no further degree of freedom -- would have to be orthogonal to these as well.
\end{remark}

\begin{example}\label{ex:signal}
	1.) If the signal is of the form
	\begin{equation*}
	S_\theta=\sum_{i=1}^D \theta_i \phi_i,
	\end{equation*}
	where $\phi_1,\ldots,\phi_D \in \L^2\big([0,\infty);\R^N\big)$ are $1$-periodic and orthonormal with respect to $\B^{(\theta,T)}_{\sigma\sigma^\top}$, we have
	\[
	\cI_{(\theta,T)}(t)=\begin{pmatrix}
	t \cdot 1_{D\times D} & -\frac{t^2}{2T^2}\left(\sum_{j=1}^D\theta_j \B^{(\theta,T)}_{\sigma\sigma^\top}[\phi_i,\phi_j']\right)_{i=1,\ldots,D} \\
	\cdots & \frac{t^3}{3T^4}\sum_{i,j=1}^D\theta_i\theta_j \B^{(\theta,T)}_{\sigma\sigma^\top}[\phi_i',\phi_j']
	\end{pmatrix}
	\]
	which is invertible for all $t\in(0,\infty)$ whenever
	\begin{equation}\label{eq:exsignal1}
	\frac43 \sum_{i,j=1}^D\theta_i\theta_j \B^{(\theta,T)}_{\sigma\sigma^\top}[\phi_i',\phi_j'] \neq \sum_{i=1}^D\left( \sum_{j=1}^D \theta_j \B^{(\theta,T)}_{\sigma\sigma^\top}[\phi_i,\phi_j']\right)^2.
	\end{equation}
	Similarly,
	\[
	\cI_{(\theta,T)}'(t)=\begin{pmatrix}
	1_{D\times D} & -tT^{-2}\left(\sum_{j=1}^D\theta_j \B^{(\theta,T)}_{\sigma\sigma^\top}[\phi_i,\phi_j']\right)_{i=1,\ldots,D} \\
	\cdots & t^2 T^{-4}\sum_{i,j=1}^D\theta_i\theta_j \B^{(\theta,T)}_{\sigma\sigma^\top}[\phi_i',\phi_j']
	\end{pmatrix}
	\]
	is invertible for all $t\in(0,\infty)$ whenever
	\begin{equation}\label{eq:exsignal2}
	\sum_{i,j=1}^D\theta_i\theta_j \B^{(\theta,T)}_{\sigma\sigma^\top}[\phi_i',\phi_j'] \neq \sum_{i=1}^D\left( \sum_{j=1}^D \theta_j \B^{(\theta,T)}_{\sigma\sigma^\top}[\phi_i,\phi_j']\right)^2.
	\end{equation}
	If $\B^{(\theta,T)}_{\sigma\sigma^\top}$ is positive definite, part 3.) of Remark \ref{rem:Fischer} gives the condition
	\begin{equation*}
		S_\theta'=\sum_{i=1}^D\theta_i\phi_i' \neq \sum_{i,j=1}^D \theta_j \B^{(\theta,T)}_{\sigma\sigma^\top}[\phi_i,\phi_j']\phi_i
	\end{equation*}
	for invertibility of both $\cI_{(\theta,T)}(t)$ and $\cI_{(\theta,T)}'(t)$.
	
	2.) For $M=N$ let $\sigma\equiv 1_{N\times N}$, then $\B^{(\theta,T)}_{\sigma\sigma^\top}$ is just the standard $\L^2$-inner product with respect to Lebesgue's measure. If $N=1$, $D=2d$ with $d\in\N$, and the signal has a finite Fourier expansion
	\[
	S_\theta(s)=\sum_{k=1}^d \sqrt{2}\left( \theta_k\sin(2k\pi s) + \theta_{d+k}\cos(2k\pi s) \right) \quad \text{for all $s\in[0,\infty)$,}
	\]
	it is both of the type from the first part of this example and of the type introduced in part 3.)\ of Example \ref{ex:signal0} (so in particular it satisfies (S1) - (S4)). Elementary calculations show that the conditions \eqref{eq:exsignal1} and \eqref{eq:exsignal2} then become
	\[
	\sum_{k=1}^dk(\theta_k^2+\theta_{k+d}^2) \neq \alpha \sum_{k=1}^dk^2\theta_{k+d}^2 \quad \text{for all $\alpha\in\{3,4\}$.}
	\]
	If for example there are no $\cos$-terms involved, i.e.\ $\theta_{d+1}= \ldots=\theta_D=0$, these inequalities are valid for all $(\theta_1,\ldots,\theta_d) \neq 0$.
\end{example}

Having introduced all relevant objects and assumptions, and having illustrated them by examples, we can now give the main result of this section.

\begin{thm}[Local Asymptotic Normality for $\cE_Z$]\label{thm:LAN}
	Grant all of the hypotheses (A1'), (A2'), (A3), (A4), (S1) - (S4) and (S5') and fix $(\theta,T) \in \Theta\times(0,\infty)$. Set
	\[
	\delta_n := \begin{pmatrix}
	n^{-1/2} & 0 & \cdots & 0 \\
	0 & \ddots &\ddots & \vdots \\
	\vdots &\ddots&n^{-1/2}&0 \\
	0 &\cdots&0& n^{-3/2} \\
	\end{pmatrix}
	\in \R^{(D+1)\times(D+1)} \quad \text{for all $n\in\N$,} 
	\]
	and fix any bounded sequence $(h_n)_{n \in \N} \subset \R^{D+1}$. Then $\Q^{(\theta,T)}$-almost surely we have
	\begin{equation}\label{eq:thmLAN1}
	\log\frac{d\Q^{(\theta,T)+\delta_n h_n}|_{\cG_n}}{d\Q^{(\theta,T)}|_{\cG_n}}=h_n^\top \cS^{(\theta,T)}_n-\frac12 h_n^\top \cI_{(\theta,T)}h_n + o_{\Q^{(\theta,T)}}(1), \quad \text{as $n\to\infty$,}
	\end{equation}
	with Fisher Information
	\[
		\cI_{(\theta,T)}=\begin{pmatrix}
		\left(\B^{(\theta,T)}_{\sigma\sigma^\top}[ \d_{\theta_i}S_\theta,\d_{\theta_j}S_\theta]\right)_{i,j=1,\ldots,D} & -\frac12T^{-2} \left(\B^{(\theta,T)}_{\sigma\sigma^\top}[\d_{\theta_i}S_\theta,S_\theta']\right)_{i=1,\ldots,D} \\
		\cdots & \frac13T^{-4} \B^{(\theta,T)}_{\sigma\sigma^\top}[S_\theta',S_\theta']
		\end{pmatrix}.
	\]
	and score
	\begin{equation*}
	\cS^{(\theta,T)}_n=\delta_n \int_0^n \left((\sigma\sigma^\top)^{-1/2}(\eta_s) \dot{S}_{(\theta,T)}(s)\right)^\top d\tilde B^{(\theta,T)}_s \quad \text{for all $n\in\N$} 
	\end{equation*}
	such that weak convergence
	\begin{equation*}
	\cL\left(\cS^{(\theta,T)}_n\middle|\Q^{(\theta,T)}\right) \xrightarrow{n\to\infty} \cN\left(0,\cI_{(\theta,T)}\right)
	\end{equation*}
	holds.
\end{thm}

\begin{proof}[Proof of Theorem \ref{thm:LAN}]
	We fix $(\theta,T)\in\Theta\times(0,\infty)$, and in order to reduce notational complexity we drop corresponding indices whenever there is no risk of ambiguity: We write $\Q:=\Q^{(\theta,T)}$, $\tilde B:=\tilde B^{(\theta,T)}$ (see \eqref{eq:BMZ}), $\cS_n:=\cS^{(\theta,T)}_n$, $\cI:=\cI_{(\theta,T)}$, $\cI(t):=\cI_{(\theta,T)}(t)$ for all $t\in[0,\infty)$ (see \eqref{eq:matrix}), and $\B:=\B^{(\theta,T)}_{\sigma\sigma^\top}$ (see \eqref{eq:Q}). Moreover, we set
	\begin{equation}\label{eq:thetan}
	(\theta_n,T_n):=(\theta,T)+\delta_n h_n  \quad \text{for all $n\in\N$.}
	\end{equation}
	We now proceed to give the proof, divided into several steps.
	
	1.) The main idea is to introduce a time step size $t\in(0,\infty)$ into the log-likelihood ratio and then interpret
	\[
	\left(\log\frac{d\Q^{(\theta,T)+\delta_n h_n}|_{\cG_{tn}}}{d\Q^{(\theta,T)}|_{\cG_{tn}}}\right)_{t\in[0,\infty)}, \quad n \in \N,
	\]
	as a sequence of continuous-time stochastic processes. Splitting them into several parts and applying Lemma \ref{lem:LAN} together with tools from continuous-time martingale theory will eventually lead to the desired quadratic expansion. Indeed, adding and subtracting the term $\dot{S}_{(\theta,T)}(s)\delta_n h_n$ to the difference of the signals yields
	\begin{align*}
		\log\frac{d\Q^{(\theta,T)+\delta_n h_n}|_{\cG_{tn}}}{d\Q^{(\theta,T)}|_{\cG_{tn}}}
		&= \int_0^{tn} \big((\sigma\sigma^\top)^{-1/2}(\eta_s)\big(S_{(\theta_n,T_n)}-S_{(\theta,T)}\big)(s)\big)^\top d\tilde B_s \\
		&\quad - \frac{1}{2}\int_0^{tn} \big(S_{(\theta_n,T_n)}-S_{(\theta,T)}\big)^\top(s) \big(\sigma\sigma^\top(z)\big)^{-1} \big(S_{(\theta_n,T_n)}-S_{(\theta,T)}\big)(s)ds \\
		&= h_n^\top \left(\delta_n \int_0^{tn} \big((\sigma\sigma^\top)^{-1/2}(\eta_s) \dot{S}_{(\theta,T)}(s)\big)^\top d\tilde B_s\right) \\
		& \quad - \frac12 h_n^\top \left( \delta_n \int_0^{tn} \dot{S}_{(\theta,T)}(s)^\top \big(\sigma\sigma^\top(\eta_s)\big)^{-1} \dot{S}_{(\theta,T)}(s) ds \, \delta_n \right)h_n  \\
		& \quad + \int_0^{tn}\left((\sigma\sigma^\top)^{-1/2}(\eta_s)\big(S_{(\theta_n,T_n)}-S_{(\theta,T)}-\dot{S}_{(\theta,T)}\delta_n h_n\big)(s)\right)^\top d\tilde B_s \\
		& \quad - \frac12\int_0^{tn} \big(S_{(\theta_n,T_n)}-S_{(\theta,T)}-\dot{S}_{(\theta,T)}\delta_n h_n\big)^\top (s) \big(\sigma\sigma^\top(\eta_s)\big)^{-1} \\ 
		&\hspace{75mm}\big(S_{(\theta_n,T_n)}-S_{(\theta,T)}-\dot{S}_{(\theta,T)}\delta_n h_n\big)(s) ds \\
		& \quad - \int_0^{tn} \big(S_{(\theta_n,T_n)}-S_{(\theta,T)}-\dot{S}_{(\theta,T)}\delta_n h_n\big)^\top (s) \big(\sigma\sigma^\top(\eta_s)\big)^{-1} \big(\dot{S}_{(\theta,T)}\delta_n h_n \big) ds \\
		&=: h_n^\top \cS_n(t)-\frac12 h_n^\top \cI_n(t)h_n+R_n(t)-\frac12 U_n(t)-V_n(t),
	\end{align*}
	and in order to prove the Theorem, we will study convergence in distribution of $\cS_n(t)$ for $n \to \infty$ and show almost sure convergence of $\cI_n(1)$ to $\cI=\cI(1)$. Finally, we show that $R_n(t)$, $U_n(t)$, and $V_n(t)$ converge to zero in probability. 
	
	2.) For any fixed $n\in\N$ the process
	\[
	M_n:=(\cS_n(t))_{t\in[0,\infty)}=\Big(\delta_n \int_0^{tn}\big((\sigma\sigma^\top)^{-1/2}(\eta_s) \dot{S}_{(\theta,T)}(s)\big)^\top d\tilde B_s\Big)_{t\in[0,\infty)}
	\]
	is obviously an $\R^{D+1}$-valued local martingale with respect to $\Q$. In order to determine its weak limit for $n \to \infty$ in the Skorohod space $\cD\big([0,\infty);\R^{D+1}\big)$, we study its quadratic variation process $\langle M_n\rangle:=(\langle M_n\rangle_t)_{t\in[0,\infty)}$ with
	\[
	\langle M_n\rangle_t:=\begin{pmatrix}
	\inpro{M_n^{(1)}}{M_n^{(1)}}_t & \cdots & \inpro{M_n^{(1)}}{M_n^{(D+1)}}_t \\
	\vdots & \ddots & \vdots \\
	\inpro{M_n^{(D+1)}}{M_n^{(1)}}_t & \cdots & \inpro{M_n^{(D+1)}}{M_n^{(D+1)}}_t
	\end{pmatrix}
	\in \R^{(D+1)\times(D+1)}.
	\]
	As follows from basic stochastic calculus, $\langle M_n\rangle$ is equal to $(\cI_n(t))_{t\in[0,\infty)}$. Consequently, for $i,j \in \{1,\ldots,D\}$ we have
	\begin{align*}
		\inpro{M_n^{(i)}}{M_n^{(j)}}_t &= \frac{1}{n} \int_0^{tn} \left(\d_{\theta_i}S_{(\theta,T)}(s)\right)^\top \left(\sigma\sigma^\top(\eta_s)\right)^{-1}  \d_{\theta_j}S_{(\theta,T)}(s)ds \\
		&= t \cdot \frac{1}{tn} \int_0^{tn} \left(\d_{\theta_i}S_\theta(s/T)\right)^\top \left(\sigma\sigma^\top(\eta_s)\right)^{-1}  \d_{\theta_j}S_\theta(s/T)ds,
	\end{align*}
	and due to the periodicity of $S_\theta$ and by part (ii) of Lemma \ref{lem:LAN} with $g=\sigma\sigma^\top$ and $k=0$, this expression converges to
	\[
	t\cdot \B[ \d_{\theta_i}S_\theta,\d_{\theta_j}S_\theta] = \cI_{i,j}(t)
	\]
	$\Q$-almost surely for $n \to \infty$. Since
	\begin{equation*}\label{eq:Tderivative}
	\d_T S_{(\theta,T)}(s)=\d_T S_\theta(s/T)=-sT^{-2}S'_\theta(s/T) \quad \text{for all $s \in(0,\infty)$},
	\end{equation*}
	the same argument with $k=1$ yields
	\begin{align*}
		\inpro{M_n^{(i)}}{M_n^{(D+1)}}_t&=\inpro{M_n^{(D+1)}}{M_n^{(i)}}_t \\
		&= \frac{1}{n^2} \int_0^{tn} \left(\d_{\theta_i}S_{(\theta,T)}(s)\right)^\top \left(\sigma\sigma^\top(\eta_s)\right)^{-1}  \d_T S_{(\theta,T)}(s)ds \\
		&= \frac{-t^2}{2T^2} \cdot \frac{1}{\frac12(tn)^2} \int_0^{tn} s \cdot \left(\d_{\theta_i}S_\theta(s/T)\right)^\top \left(\sigma\sigma^\top(\eta_s)\right)^{-1}  S'_\theta(s/T)ds \\
		& \xrightarrow{n \to \infty}  \frac{-t^2}{2T^2}\cdot  \B[ \d_{\theta_i}S_\theta,S'_\theta] =\cI_{i,{D+1}}(t)= \cI_{{D+1},i}(t)
	\end{align*}
	$\Q$-almost surely, and analogously (with $k=2$)
	\begin{align*}
		\inpro{M_n^{(D+1)}}{M_n^{(D+1)}}_t &= \frac{1}{n^3} \int_0^{tn} \left(\d_T S_{(\theta,T)}(s)\right)^\top \left(\sigma\sigma^\top(\eta_s)\right)^{-1}  \d_T S_{(\theta,T)}(s)ds \\
		&= \frac{t^3}{3T^4} \cdot \frac{1}{\frac13(tn)^3} \int_0^{tn} s^2 \cdot S'_\theta(s/T)^\top \left(\sigma\sigma^\top(\eta_s)\right)^{-1} S'_\theta(s/T) ds \\
		& \xrightarrow{n \to \infty} \frac{t^3}{3T^4} \cdot \B[ S'_\theta,S'_\theta] =\cI_{D+1,D+1}(t)
	\end{align*}
	$\Q$-almost surely. In other words,
	\[
	\langle M_n\rangle_t \xrightarrow{n \to \infty} \cI(t) \quad \text{$\Q$-almost surely for all $t\in[0,\infty)$,}
	\]
	and hence the Martingale Convergence Theorem \cite[Corollary VIII.3.24]{Jacod} implies weak convergence
	\begin{equation}\label{eq:weakconvergence}
	\cL(M_n|\Q) \xrightarrow{n\to\infty} \cL(M|\Q) \quad \text{in $\cD\big([0,\infty);\R^{D+1}\big)$}
	\end{equation}
	to some limit martingale $M=(M(t))_{t\in[0,\infty)}$ with quadratic variation process $\langle M\rangle=(\cI(t))_{t\in [0,\infty)}$.\footnote{To be exact, $M$ is actually defined on some arbitrary probability space, but in order to avoid making things more complicated than necessary, we assume without loss of generality that $M$ is in fact defined on (a standard extension of) the same probability space as the sequence $(M_n)_{n\in\N}$.} As noted in Remark \ref{rem:Fischer}, $\cI'(t)$ is symmetric and non-negative definite, so it possesses a square root $\sqrt{\cI'(t)} \in \R^{(D+1)\times(D+1)}$. By (S5'), $\cI'(t)$ is invertible and hence $\sqrt{\cI'(t)}$ is invertible as well. Thus, the Representation Theorem \cite[Theorem II.7.1]{Ikeda} yields that $M$ can be expressed as
	\[
	M(t)=\int_0^t \sqrt{\cI'(s)} dB'_s \quad \text{for all $t\in[0,\infty)$}
	\]
	with some $(D+1)$-dimensional Brownian Motion $B'$. Together with \eqref{eq:weakconvergence}, this also implies weak convergence
	\[
	\cL(M_n(t)|\Q) \xrightarrow{n\to\infty} \cL(M(t)|\Q)=\cN\left(0,\int_0^t \cI'(s)ds\right)=\cN\left(0,\cI(t)\right)
	\]
	for all $t\in[0,\infty)$. In particular, choosing $t=1$ yields weak convergence of the score
	\[
	\cL(\cS_n|\Q)=\cL(M_n(1)|\Q) \xrightarrow{n\to\infty} \cN(0,\cI(1))=\cN(0,\cI),
	\]	
	which completes this step of the proof.
	
	3.) In the second step, we have shown on the fly that 
	\[
	\cI_n(1) = \langle M_n\rangle_1 \xrightarrow{n\to\infty} \langle M\rangle_1=\cI(1)
	\]
	$\Q$-almost surely.
	
	4.) It remains to show convergence to zero in $\Q$-probability of the remainder terms $R_n(t)$, $U_n(t)$, and $V_n(t)$ introduced at the very beginning of this proof. Therefore, we consider the sequence $(R_n)_{n\in\N}$ of the local $\Q$-martingales
	\[
	(R_n(t))_{t\in[0,\infty)}=\left(\int_0^{tn}\Big((\sigma\sigma^\top)^{-1/2}(\eta_s)\big(S_{(\theta_n,T_n)}-S_{(\theta,T)}-\dot{S}_{(\theta,T)}\delta_n h_n\big)(s)\Big)^\top d\tilde B_s\right)_{t\in[0,\infty)}.
	\]
	Their quadratic variation processes are obviously given by $(U_n(t))_{t\in[0,\infty)}$. Exploiting the uniform ellipticity assumption (A2'), we can estimate the quadratic variation by
	\begin{align}\label{eq:ABC}
		\begin{split}
			\langle R_n \rangle_t &= \int_0^{tn} \big(S_{(\theta_n,T_n)}-S_{(\theta,T)}-\dot{S}_{(\theta,T)}\delta_n h_n\big)^\top (s) \big(\sigma\sigma^\top(\eta_s)\big)^{-1} \big(S_{(\theta_n,T_n)}-S_{(\theta,T)}-\dot{S}_{(\theta,T)}\delta_n h_n\big)(s) ds \\
			&\le \sigma_0^{-1} \int_0^{tn} \abs{S_{(\theta_n,T_n)}-S_{(\theta,T)}-\dot{S}_{(\theta,T)}\delta_n h_n }^2 ds\\
			&=\sigma_0^{-1} \int_0^{tn} \abs{S_{(\theta_n,T_n)}-S_{(\theta,T)}-D_\theta S_{(\theta,T)}(\theta_n-\theta)- \d_T S_{(\theta,T)}(s) (T_n-T)}^2 ds.
		\end{split}
	\end{align}
	Note that this upper bound is entirely deterministic. In order to prove that it in fact converges to zero, we will separate the dependence on the parameters $\theta$ and $T$ in such a way that we can use the periodicity and (S1) - (S4) efficiently. This can be achieved by continuing the inequality \eqref{eq:ABC} with
	\begin{align*}
		\langle R_n \rangle_t &\le 3 \sigma_0^{-1} \bigg( \int_0^{tn} \abs{S_{(\theta_n,T_n)}(s)-S_{(\theta,T_n)}(s)- D_\theta S_{(\theta,T_n)}(s)(\theta_n-\theta)}^2ds \\
		& \hspace{11,3mm} + \int_0^{tn} \abs{\big(D_\theta S_{(\theta,T_n)}-D_\theta S_{(\theta,T)}(s)\big)(\theta_n-\theta)}^2ds \\
		& \hspace{11,3mm} + \int_0^{tn} \abs{S_{(\theta,T_n)}(s)-S_{(\theta,T)}(s)-\d_T S_{(\theta,T)}(s) (T_n-T)}^2ds \bigg) \\
		& =: 3 \sigma_0^{-1} (A_n+B_n+C_n).
	\end{align*}
	We will treat convergence of $A_n$, $B_n$, and $C_n$ step for step. For this purpose, set $H:=\sup_{n\in\N}\abs{h_n}$ and note that due to \eqref{eq:thetan} we have
	\[
	\abs{\theta_n-\theta} \le H n^{-1/2} \quad \text{and} \quad \abs{T_n-T} \le H n^{-3/2}
	\]
	for all $n\in\N$.
	
	Starting with $A_n$, we observe that for sufficiently large $n\in\N$ we have $T_n\in[T/2,2T]$ and thus
	\begin{align*}
		A_n 
		&\le \left(\frac{tn}{T_n}+1\right)\int_0^{T_n} \abs{S_{(\theta_n,T_n)}(s)-S_{(\theta,T_n)}(s)- D_\theta S_{(\theta,T_n)}(s)(\theta_n-\theta)}^2ds \\
		&= \left(\frac{tn}{T_n}+1\right) \abs{\theta_n-\theta}^2 \int_0^{T_n}\!\abs{\frac{S_{(\theta_n,T_n)}(s)-S_{(\theta,T_n)}(s)-D_\theta S_{(\theta,T_n)}(s)(\theta_n-\theta)}{\abs{\theta_n-\theta}} }^2\! ds \\
		&\le \left(\frac{tn}{T/2}+1\right) H^2n^{-1}\int_0^{2T} \!\abs{\frac{S_{(\theta_n,T_n)}(s)-S_{(\theta,T_n)}(s)-D_\theta S_{(\theta,T_n)}(s)(\theta_n-\theta)}{\abs{\theta_n-\theta}} }^2\! ds,
	\end{align*}
	where the factor in front of the integral is obviously convergent. Using the $\L^2$-continuity condition (S3) and a simple application of the mean value theorem (compare Lemma 3.18 of \cite{ICHDiss}), one sees that the integral itself tends to zero.
	
	Next, using the H{\"o}lder condition (S4), we obtain for sufficiently large $n\in\N$ that
	\begin{align*}
		B_n &\le \abs{\theta_n-\theta}^2 \int_0^{tn} \abs{D_\theta S_{(\theta,T_n)}(s)-D_\theta S_{(\theta,T)}(s)}^2ds \\
		&\le H^2n^{-1} \left(\int_0^{t_0} \abs{D_\theta S_{(\theta,T_n)}(s)-D_\theta S_{(\theta,T)}(s)}^2ds + C (tn)^\beta \abs{T_n-T}^\alpha \right) \\
		&\le H^2n^{-1} \int_0^{t_0} \abs{\dot S_{(\theta,T_n)}(s)-\dot S_{(\theta,T)}(s)}^2ds + C H^{2+\alpha} t^\beta n^{\beta-(1+3\alpha/2)}.
	\end{align*}
	The particular conditions on $\alpha$ and $\beta$ from (S4) make the second summand vanish for $n\to\infty$, while the first summand converges to zero because of (S3).
	
	In order to estimate $C_n$, we make explicit use of the $C^2$-property (S1) which is readily translated into the condition that the mapping
	\[
	(0,\infty)\ni T \mapsto S_{(\theta,T)}(s)
	\]
	is twice continuously differentiable for any fixed $s\in(0,\infty)$. Consequently, for every $s \in (0,\infty)$ and any $i\in\{1,\ldots,N\}$ Taylor expansion with the Lagrange form of the remainder provides a $\rho_i=\rho_i(s,\theta,T,T_n,h_n)$ between $T$ and $T_n$ such that for sufficiently large $n \in \N$ we can infer that
	\begin{align*}
		\abs{S_{(\theta,T_n)}(s)-S_{(\theta,T)}(s)-(T_n-T) \d_T S_{(\theta,T)}(s)}^2
		&= \sum_{i=1}^N \left(\frac12(T_n-T)^2 \d_T^2S^{(i)}_{(\theta,T)}(s)_{|_{T=\rho_i}}\right)^2 \\
		&=\frac14 \left(T_n-T\right)^4\sum_{i=1}^N\left(\frac{s^2}{\rho_i^4} \left(S_\theta^{(i)}\right)''(s/\rho_i)+\frac{2s}{\rho_i^3} \left(S_\theta^{(i)}\right)'(s/\rho_i)\right)^2 \\
		&\le \frac14 H^4 n^{-6} 2N \left[\left( s^2\frac{\norm{S_\theta''}_\infty}{(T-n^{-3/2}H)^4}\right)^2 +\left(s\frac{2\norm{S_\theta'}_\infty}{(T-n^{-3/2}H)^3}\right)^2\right] \\
		&\le \cst n^{-6}(s^4+s^2)
	\end{align*}
	for some positive constant not depending on $s$ or $n$. Integrating yields
	\[
	C_n \le \cst n^{-6}\int_0^{tn}(s^4+s^2)ds
	\]
	and hence $C_n$ vanishes for $n \to \infty$.
	
	So far, we have shown that the sequence of random variables $(U_n(t))_{n\in\N}$ not only vanishes in probability under $\Q$ for $n \to \infty$, but is even bounded by a deterministic sequence which goes to zero. Therefore,
	\begin{equation}\label{eq:remaindertozero}
	\E_\Q[ R_n(t)^2] = \E_\Q[\langle R_n\rangle_t] = \E_\Q[U_n(t)] \xrightarrow{n \to \infty} 0,
	\end{equation}
	and in particular, $R_n(t)$ also vanishes in probability under $\Q$ for $n \to \infty$. Finally, the same is true for the last remainder variable $V_n(t)$, as by the Cauchy-Schwarz inequality we get that
	\begin{equation}\label{eq:remaindertozero2}
	\abs{V_n(t)}^2\le U_n(t) h_n^\top \cI_n(t)h_n \le U_n(t) H^2 \abs{\cI_n(t)}\xrightarrow{n \to \infty} 0, 
	\end{equation}
	since $\cI_n(t)$ converges and $U_n(t)$ goes to zero. Taking $t=1$ completes the proof.
\end{proof}

\begin{remark}
	The convergence in probability for $n\to\infty$ of the remainder terms $R_n(t)$, $U_n(t)$, and $V_n(t)$ (which determine the term $o_{\Q^{(\theta,T)}}(1)$ in \eqref{eq:thmLAN1}) is in fact even uniform with respect to $t\in[0,t_0]$ for every $t_0\in(0,\infty)$. For $U_n(t)$ this is clear, since it only increases with $t$. Using the Burkholder-Davis-Gundy inequality, the estimation \eqref{eq:remaindertozero} can be improved to
	\[
	\E_\Q\Big[ \sup_{t\in[0,t_0]} \abs{R_n(t)}^2\Big] \le 4 \E_\Q[\langle R_n\rangle_{t_0}] = 4 \E_\Q[ U_n(t_0)] \xrightarrow{n \to \infty} 0,
	\]
	which also takes care of $R_n(t)$. For $V_n(t)$ we notice that the bound given in \eqref{eq:remaindertozero2} only depends on $t$ via $\cI_n(t)$ and $U_n(t)$ which are both non-decreasing with respect to $t$.
\end{remark}

\begin{remark}
	In the one-dimensional case $M=N=1$, variants of Theorem \ref{thm:LAN} are already known in the literature, where shape and periodicity are treated separately and one of them is assumed to be known. A detailed contextualization is provided in Remark 2.6 and Examples 2.7 and 2.8 of \cite{ICH}.
\end{remark}

\noindent\textbf{Acknowledgements.} The author would like to thank Reinhard H{\"o}pfner for fruitful discussions and helpful remarks and suggestions.

\end{document}